\documentclass[reqno,10pt]{amsart}
\usepackage{pdfsync}
\usepackage{ifthen}
\usepackage{amsmath,paralist}
\usepackage{amsthm}
\usepackage{graphicx,eucal}
\usepackage{graphpap,latexsym,epsf,bm}
\usepackage{color}
\usepackage{hyperref}
\usepackage{amssymb,mathrsfs,enumerate}
\numberwithin{equation}{section}
\numberwithin{figure}{section}

\newtheorem{theorem}{Theorem}[section]

\newtheorem{lemma}[theorem]{Lemma}
\newtheorem{proposition}[theorem]{Proposition}
\newtheorem{definition}[theorem]{Definition}%
\newtheorem{remark}[theorem]{Remark}%

\newcommand{\R}{\mathbb{R}} % real numbers

\newcommand{\calI}{\mathcal{I}}

\newcommand{\dd}{\,\mathrm{d}}
\newcommand{\rmd}{\mathrm{d}}
\newcommand{\eps}{\varepsilon}

\newcommand{\weakto}{\rightharpoonup}
\newcommand{\Restr}[1]{\lower2pt\hbox{$|_{#1}$}}

\newcommand{\en}[3]{\mathcal{E}(#1, #2, #3)}
\newcommand{\ene}[2]{\mathcal{E}(#1, #2)}

\newcommand{\BV}{\ensuremath{\mathrm{BV}}}
\newcommand{\AC}{\mathrm{AC}}

\newcommand{\vvmV}[3]{
  \fre_{#1}(#2)
  \ifthenelse{\equal{#3}{}} {} {\|#3\|} }

\newcommand{\FNormname}[1]{\mathscr{P}}

\newcommand{\Cost}[4]{\ifthenelse{
    \equal{#1}\Phi}
  {\Phi(#4-#3)}
  {\Delta_{#1_{#2}}(#3,#4)}
}

\newcommand{\foraa}{\text{for a.a.\,}}

\newcommand{\down}{\downarrow}

\newcommand{\rmC}{\mathrm{C}}
\newcommand{\rmD}{\mathrm{D}}

\newcommand{\Vtaue}[1]{\mathrm{V}_{\kern-1pt\tau,\eps}^{#1}}
\newcommand{\Xitaue}[1]{\Xi_{\kern-1pt\tau,\eps}^{#1}}

\newcommand{\piecewiseConstant}[2]{\overline{\mathrm{#1}}_{\kern-1pt#2}}

\newcommand{\underpiecewiseConstant}[2]{\underline{\mathrm{#1}}_{\kern-1pt#2}}

\newcommand{\piecewiseLinear}[2]{\mathrm{#1}_{\kern-1pt#2}}

\newcommand{\pwM}[2]{\widetilde{\mathrm{#1}}_{\kern-1pt#2}}
 \def\trait #1 #2 #3 {\vrule width #1pt height #2pt depth #3pt}

\newcommand{\sfz}{\mathsf{z}}
\newcommand{\sfu}{\mathsf{u}}

\newcommand{\sft}{\mathsf{t}}
\newcommand{\sfq}{\mathsf{q}}

\newcommand{\fre}{\mathfrak e}

\newcommand{\Diss}[2]{
  \ifthenelse      {   \equal{#1}{V}  }         {\Psiv}
  {
    \ifthenelse      {   \equal{#1}{B}  }         {\Psiz}
    {
      \ifthenelse      {   \equal{#1}{\Bo}  }     {\Psiz}
      {\Psi_{\kern-1pt#1}}
      }
    }
  ^{#2}
}

\newcommand{\Dnorm}[1]{\Psiz_{\kern-2pt \scriptscriptstyle\land}(#1)}
\newcommand{\Dnormname}{\Psiz_{\kern-2pt \scriptscriptstyle\land}}

\newcommand{\cE}{\mathcal{E}}
\newcommand{\calE}{\cE}
\newcommand{\calM}{\mathcal{M}}

\def\nchi{{\raise.3ex\hbox{$\chi$}}}
\newcommand{\Bo}{{}}

\newcommand{\rmV}{\mathrm V}
\newcommand{\scalarV}[3]{\rmV_{\!#2}
  \ifthenelse{\equal{#3}{}}{}{(#3)}}
\newcommand{\scalarmu}[2]{\mu}
\newcommand{\scalarmuco}[2]{\mu_{\rm d}}
\newcommand{\scalarmuj}[2]{\mu_{\rm J}}
\newcommand{\scalarmul}[2]{\mu_{\mathscr L}}
\newcommand{\scalarmuc}[2]{\mu_{\rm C}}

\newcommand{\Psiv}{\Phi}
\newcommand{\Psiz}{\Psi}

\newcommand{\calR}{\mathcal{R}}
\newcommand{\calV}{\mathcal{V}}
\newcommand{\calW}{\mathcal{W}}
\newcommand{\calQ}{\mathcal{Q}}
\newcommand{\enet}[3]{\mathcal{E}(#1,#2,#3)}
\newcommand{\vpot}[3]{\calV_{\mathsf{#1}}(#2;#3)}
\newcommand{\vcof}[3]{\mathbb{V}_{\mathsf{#1}}(#2){#3}}
\newcommand{\vcofinv}[3]{\mathbb{V}_{\mathsf{#1}}(#2)^{-1}{#3}}
\newcommand{\vpotname}[1]{\calV_{\mathsf{#1}}}
\newcommand{\vcofname}[1]{\mathbb{V}_{\mathsf{#1}}}
\newcommand{\bbV}{\mathbb{V}}
\newcommand{\conjzspe}[4]{\calW_{\mathsf{#1},#2}^*(#3;#4)}
\newcommand{\conjz}[3]{\calW_{\mathsf{#1}}^*(#2;#3)}
\newcommand{\moreau}[3]{\calV_{\mathsf{#1}}^*(#2;#3)}
\newcommand{\conjzname}[2]{\calW_{\mathsf{#1}}^*(#2)}

\newcommand{\stab}[1]{K(#1)}

\newcommand{\thn}[1]{\theta_{\mathsf{#1}}}

\newcommand{\Me}[5]{\mathcal{M}_{#1}(#2,#3, #4,#5)}
\newcommand{\Mo}[4]{\mathcal{M}_{0}(#1,#2, #3,#4)}
\newcommand{\Mored}[4]{\mathcal{M}_{0}^{\mathrm{red}}(#1,#2, #3,#4)}
\newcommand{\moredname}{\mathcal{M}_{0}^{\mathrm{red}}}
\newcommand{\Mered}[4]{\mathcal{M}_{\eps}^{\mathrm{red}}(#1,#2, #3,#4)}
\newcommand{\Sign}{\mathop{\mathrm{Sign}}}
\newcommand{\calU}{\mathcal{U}}
\newcommand{\calZ}{\mathcal{Z}}

\newcommand{\contset}[1]{\Sigma(#1)}

\newcommand{\pbv}{\mathrm{pBV}}
\newcommand{\rgm}[4]{\mathrm{#1}_{\mathsf{#2}}\mathrm{#3}_{\mathsf{#4}}}

\begin{document}

\title[Balanced-Viscosity solutions for multi-rate systems]
{Balanced-Viscosity solutions for multi-rate systems}

\date{August 28, 2014}

\author{Alexander Mielke}
\address{Weierstra\ss-Institut,
  Mohrenstra\ss{}e 39,  D--10117 Berlin {\upshape and}  Institut f\"ur
  Mathematik, Humboldt-Universit\"at zu
  Berlin, Rudower Chaussee 25, D--12489 Berlin (Adlershof), Germany.}
\email{alexander.mielke\,@\,wias-berlin.de}

\author{Riccarda Rossi}
\address{DICATAM -- Sezione di Matematica, Universit\`a di
  Brescia, via Valotti 9, I--25133 Brescia, Italy.}
\email{riccarda.rossi\,@\,unibs.it}

\author{Giuseppe Savar\'e}
\address{Dipartimento di Matematica ``F.\ Casorati'', Universit\`a di
  Pavia.  Via Ferrata, I--27100 Pavia, Italy.}
\email{giuseppe.savare\,@\,unipv.it}

\thanks{A.M.\ has been partially
  supported by the  ERC grant no.267802 AnaMultiScale.
  R.R.\ and G.S.\ have been partially supported by a MIUR-PRIN'10-11 grant
  for the project ``Calculus of Variations''. R.R.\  also acknowledges
  support from the  Gruppo Nazionale per  l'Analisi Matematica, la
  Probabilit\`a  e le loro Applicazioni (GNAMPA) of the Istituto Nazionale di Alta Matematica (INdAM)}

\begin{abstract}
%We consider an ODE  system consisting of the rate-independent flow rule
Several mechanical systems are modeled by
the static momentum balance for the displacement $u$
coupled with
a rate-independent  flow rule for some internal variable $z$.
We consider a  class of abstract systems  of ODEs which have the same structure, albeit in a finite-dimensional setting,
and
regularize both the static equation and the rate-independent flow rule
 by adding  viscous dissipation  terms with coefficients $\eps^\alpha$ and $\eps$, where $0<\eps \ll 1$ and $\alpha>0$ is a fixed parameter. Therefore for $\alpha \neq 1$ $u$ and $z$ have different relaxation rates.

We address the vanishing-viscosity analysis as $\eps \downarrow 0$ of the viscous system. We prove that, up to a subsequence,
(reparameterized) viscous solutions converge to  a parameterized curve yielding a  \emph{Balanced Viscosity} solution to the original
rate-independent system, and providing an accurate  description of the system behavior at jumps.  We also give a reformulation of the notion of \emph{Balanced Viscosity} solution in terms of a system of
subdifferential inclusions, showing that the viscosity in $u$ and the one in $z$ are involved in the jump dynamics in different ways,  according to whether $\alpha>1$,
$\alpha=1$, and $\alpha
\in (0,1)$.
\end{abstract}
\maketitle

%\tableofcontents

\section{Introduction}
\label{s:Intro}
Several mechanical systems
 are described by  ODE or PDE systems of the type:
\begin{subequations}
\label{intro:rip-limit}
\begin{align}
\label{intro:rip-limit-1}
&
 \rmD_u \enet{t}{u(t)}{z(t)} =0  && \text{in } \calU^*,  && \foraa\, t \in (0,T),
 \\
 &
 \label{intro:rip-limit-2}
 \partial\calR_0(z'(t)) + \rmD_z \enet{t}{u(t)}{z(t)}  \ni 0 && \text{in } \calZ^*, &&  \foraa\, t \in (0,T),
 \end{align}
\end{subequations}
where $\calU$, $\calZ$ are Banach spaces, and  $\calE : [0,T] \times \calU \times \calZ \to \R$  is an energy functional.
For example,  within the ansatz of generalized standard materials, $u$ is the displacement, at equilibrium,
 while changes in the elastic behavior due to dissipative effects are
described in terms of an internal variable $z$ in some state space $\calZ$. In several mechanical phenomena \cite{Miesurvey}, dissipation due to
inertia and viscosity is negligible, and the system is governed by  rate-independent evolution,
which means that the (convex, nondegenerate) dissipation potential $\calR_0 : \calZ \to [0,\infty) $ is \emph{positively homogeneous of degree $1$}.
Thus system \eqref{intro:rip-limit-2} is invariant for time-rescalings.

It is well known that, if the map $z\mapsto \enet tuz$ is not uniformly convex, one cannot expect the existence of absolutely continuous solutions to
system \eqref{intro:rip-limit}.
% which can be in fact tackled as a rate-independent system in the \emph{sole} variable $z$ by considering the \emph{reduced} energy obtained minimizing out the variable $u$, i.e.
%\begin{equation}
%\label{rip-reduced}
% \partial\calR_0(z'(t)) + \rmD_z \calI(t,z(t)) \ni 0 \quad \text{in }\calZ^*, \qquad \foraa\, t \in (0,T) \qquad \text{with } \calI(t,z):= \min_{u \in \calU} \enet tuz\,.
%\end{equation}
This fact has motivated the development of various
weak solvability concepts
for \eqref{intro:rip-limit}, starting with the well-established notion of \emph{energetic solution}. The latter
dates back to \cite{Mie-Theil99} and was further developed in \cite{Mielke-Theil04} (see  \cite{DaFrTo05QCGN}, as well,  in the context of crack growth),
cf.\ also \cite{Miesurvey}, \cite{Miel08?DEMF}
 and the references therein.
 Despite the several good features of the energetic formulation, it is known that, in the case the energy $z \mapsto \enet tuz$ is nonconvex,
the global stability condition
   may lead to
   jumps of $z$ as a function of time that are not motivated by, or in accord with, the mechanics of the system, cf.\ e.g. the discussions
   in \cite[Ex.\,6.1]{Miel03EFME},
\cite[Ex.\,6.3]{KnMiZa07?ILMC}, and
\cite[Ex.\,1]{MRS09}.

Over the last years,
an alternative selection criterion of  mechanically feasible weak solution concepts for  the rate-independent system \eqref{intro:rip-limit}
has been developed,
moving from the finite-dimensional analysis in \cite{ef-mie06}.
 It is
 based on the interpretation of
\eqref{intro:rip-limit}
 as originating in the  vanishing-viscosity limit  of the \emph{viscous} system
 %\partial\calR_0(z'(t))+\eps \partial\vpotname{z} (z'(t)) + \rmD_z \calI(t,z(t))  \ni 0\qquad   \text{in } \calZ^*,\quad   \foraa\, t \in (0,T)
 \begin{subequations}
\label{van-visco-intro}
\begin{align}
\label{van-visc-1-intro}
&
 \rmD_u \enet{t}{u(t)}{z(t)} =0  && \text{in } \calU^*,  && \foraa\, t \in (0,T),
 \\
 &
 \label{van-visc-2-intro}
 \partial\calR_0(z'(t))  +\eps \partial\vpotname{z} (z'(t))  + \rmD_z \enet{t}{u(t)}{z(t)}  \ni 0 && \text{in } \calZ^*, &&  \foraa\, t \in (0,T),
 \end{align}
\end{subequations}
 where $\vpotname{z} : \calZ \to [0,\infty)
$  is a  dissipation potential with \emph{superlinear}  (for instance, quadratic) growth at infinity.
Observe that the existence of solutions for  the \emph{generalized gradient system} \eqref{van-visco-intro} follows from  \cite{ColliVisintin90,Colli92},
cf.\ also \cite{MRS-dne}.
This vanishing-viscosity approach leads to a notion of solution featuring  a \emph{local}, rather than \emph{global}, stability  condition
for the  description of  rate-independent
evolution, thus
avoiding ``too early'' and ``too long'' jumps.
 Furthermore, it provides an accurate description of the energetic behavior of the system at jumps, in particular highlighting how viscosity,
 neglected in the limit as $\eps \downarrow 0$, comes back into the picture and governs the jump dynamics.
 This has been demonstrated in \cite{MRS09,MRS10,mielke-rossi-savare2013}
within the frame of abstract, finite-dimensional and infinite-dimensional, rate-independent systems, and in \cite{Mielke-Zelik}
for a wide class parabolic equations with a rate-independent term.  This analysis has also been developed
in several applicative contexts, ranging from crack propagation \cite{ToaZan06, KnMiZa07?ILMC}, to plasticity \cite{DalMaso-DeSimone-Solombrino2011,DalMaso-DeSimone-Solombrino2012,BabFraMor12, FraSte},  and to  damage \cite{KnRoZa2011}, among others.

 %Aim of the vanishing-viscosity approach:
%recover solutions jumping later, and thorough energetic description of the system behavior at jumps, where viscosity intervenes.
%Originating from \cite{ef-mie06},
%there is by now a well-established literature on this topic,   \cite{MRS09,MRS10,mielke-rossi-savare2013},
  %and  of   but also   \beric other citations? \eric

In this note, we shall perform the vanishing viscosity analysis of system \eqref{intro:rip-limit} by considering the viscous approximation of  \eqref{intro:rip-limit-1}, in addition to the viscous approximation of \eqref{intro:rip-limit-2}.
%with a \emph{(possibly) different} rate.
 More precisely, we will address the  asymptotic analysis as $\eps\down 0$ of the system
\begin{subequations}
\label{intro:vv-limit}
\begin{align}
\label{intro:vv-limit-1}
&
\eps^\alpha \partial\vpotname{u} (u'(t)) +  \rmD_u \enet{t}{u(t)}{z(t)} =0 && \text{in } \calU^*, && \foraa\, t \in (0,T),
 \\
 &
 \label{intro:vv-limit-2}
 \partial\calR_0(z'(t)) + \eps \partial\vpotname{z} (z'(t)) + \rmD_z \enet{t}{u(t)}{z(t)}  \ni 0 && \text{in } \calZ^*, &&  \foraa\, t \in (0,T),
 \end{align}
\end{subequations}
where $\alpha >0$ and
$\vpotname{u}$ a quadratic dissipation potential for the variable $u$. Observe that \eqref{intro:vv-limit} models
systems with (possibly)  \emph{different} relaxation times.  In fact , the parameter $\alpha>0$ sets which of the two variables $u$ and $z$ relaxes faster to
\emph{equilibrium} and \emph{rate-independent} evolution, respectively.

Let us mention that
 the analysis developed in this paper is in the mainstream of  a series of recent papers
 focused on the coupling between rate-independent and viscous systems. First and foremost,  in
  \cite{Roub09}
a wide class of rate-independent processes in viscous solids with inertia has been tackled, while  the coupling with temperature has further been considered in  \cite{Roub10}.  In fact,   in these systems %tackled  by \textsc{Roub{\'{\i}}{\v{c}}ek}
the evolution for the internal variable $z$ is  purely rate-independent and no vanishing viscosity  is added to the equation for $z$,
  viscosity and inertia   only intervene in the evolution for the displacement $u$.   For these processes, the author
  has proposed a notion of solution of  energetic type consisting of the weakly formulated momentum equation for the displacements (and also of the weak heat equation
  in \cite{Roub10}), of an energy balance, and of a  \emph{semi-stability} condition. The latter reflects the mixed \emph{rate dependent/independent} character of the system. In
 \cite{Roub09} and  \cite{Roub13} a vanishing-viscosity analysis  (in the momentum equation) has been performed. As discussed in  \cite{Roub13} in the context of delamination, this approach leads to   \emph{local solutions} (cf.\ also \cite{Miel08?DEMF}),
 describing crack initiation (i.e., delamination) in a
 physically feasible way. % unlike energetic solutions.
  In \cite{Racca}, the vanishing-viscosity approach has also been developed in the context
  of a model for crack growth in the two-dimensional antiplane case, with a pre-assigned crack path, coupling
   a viscoelastic momentum equation  with a viscous flow rule for the crack tip; again, this procedure leads to solutions jumping later than energetic solutions. With a rescaling technique, a vanishing-viscosity analysis both in the flow rule, and in the momentum equation, has   been
   recently performed  in  \cite{DM-Scala} for perfect plasticity, recovering energetic solutions thanks to the convexity of the energy. In \cite{Scala}, the same analysis has led to
   \emph{local solutions}  for
  a delamination system.%\footnote{\beric  here more comments... \eric}

% From a different viewpoint,  has also been thoroughly analyzed by  in a series of papers.  In particular, let us mention ,
%analyzing , and, where
  With the vanishing-viscosity analysis in this paper,
   besides finding good  \emph{local}  conditions for the limit evolution,
      we
 want to  add as an  additional feature
  a thorough description of the energetic behavior of the solutions at jumps. This  shall be deduced  from an \emph{energy balance}. Moreover,
  in comparison to the aforementioned contributions  \cite{Racca,DM-Scala, Scala} a greater emphasis shall be put here on how the multi-rate character of system
\eqref{intro:vv-limit}
enters in the  description of the jump  dynamics.
% in the limit as $\eps \down 0$.
In particular, we will
  %  The latter shall also
   convey that  viscosity in $u$ and viscosity $z$
are involved in the path followed by the system at jumps in (possibly) different ways, depending on whether the  parameter $\alpha$ is strictly bigger than, or equal to, or strictly smaller than $1$.
 %  in a differentdepending on whether

%\footnote{\beric this is very rough... \eric}

To focus on this and to avoid  overburdening the paper
with technicalities,
we shall keep to a simple functional analytic setting.
%which still allows us to highlight the main ides.
 Namely, we shall consider the \emph{finite-dimensional}  and \emph{smooth}  case
%\begin{subequations}
\label{simpler-intro}
\begin{equation}
\label{simpler-intro-1}
\calU = \R^n, \qquad \calZ = \R^m, \qquad
%\end{equation}
%and confine our analysis to a \emph{smooth} energy functional
%\begin{equation}
%\label{simpler-intro-2}
 \calE \in \rmC^1 ([0,T]\times \R^{n} \times
 \R^m)\,.
\end{equation}
%\end{subequations}
Obviously, this considerably simplifies the analysis, since the difficulties attached to nonsmoothness of the energy
 and %the technicalities related
 to infinite-dimensionality are completely avoided. Still, even within such a simple setting
(where, however, we will allow for  state-dependent dissipation potentials
$\calR_0$,
$\vpotname z$, and $\vpotname u$), the key ideas
of our vanishing-viscosity approach
can be highlighted.
%Nonetheless, let us  mention that

% We postpone\footnote{\beric should we say this? \eric} to a forthcoming paper the analysis of the infinite-dimensional case.

Let us briefly summarize our results, focusing on a further simplified version of
\eqref{intro:vv-limit}. In the setting of
 \eqref{simpler-intro-1}, and with the choices
  \[
  \vpotname u (u') = \frac12 |u'|^2, \qquad
\vpotname z (z') = \frac12 |z'|^2,
\]
 system \eqref{intro:vv-limit}
 reduces to
the ODE system
\begin{subequations}
\label{eps-system-intro}
\begin{align}
\label{eq-u-intro}
&
\eps^\alpha u'(t) + \rmD_u \enet{t}{u(t)}{z(t)} =0 && \text{in } (0,T),
\\
&
\label{eq-z-intro}
 \partial\calR_0(z'(t)) + \eps z'(t) + \rmD_z \enet{t}{u(t)}{z(t)}  \ni 0 && \text{in } (0,T).
\end{align}
\end{subequations}
%with $\alpha>0$
%a strictly positive parameter.

First of all, following \cite{MRS09,MRS10,mielke-rossi-savare2013}, and along the lines of the \emph{variational} approach to
gradient flows by
\textsc{E.\ De Giorgi} \cite{Ambrosio95, AGS08}, we will pass to the limit as $\eps \down 0$ in the \emph{energy-dissipation} balance associated (and equivalent, by Fenchel-Moreau duality and the chain rule for $\calE$) to
\eqref{eps-system-intro}, namely
\begin{equation}
\label{enid-eps-expl-intro}
\begin{aligned}
&
\enet t{u(t)}{z(t)}   +
\int_s^t
 \calR_0 (z'(r)) +
 \frac{\eps}2 |z'(r)|^2 +\frac{\eps^\alpha}2 |u'(r)|^2 \dd r
 \\ & \quad +
\int_s^t
 \frac1{\eps}  \conjzname {z}{\rmD_z \enet r{u(r)}{z(r)}}
 +\frac{1}{2\eps^\alpha} |\rmD_u \enet r{u(r)}{z(r)}|^2
 \dd r
 \\
 &
 = \enet s{u(s)}{z(s)}+ \int_s^t \partial_t \enet r{u(r)}{z(r)} \dd r
 \end{aligned}
\end{equation}
for all $0 \leq s \leq t \leq T$, where
%$\mathcal{V}_{\mathsf{u}}^*$
$\mathcal{W}_{\mathsf{z}}^*$ %respectively)
is the Legendre transform of
%$\vpotname u$ (of
$\calR_0 + \vpotname z$. % respectively).
As we will see in Section \ref{s:3},  \eqref{enid-eps-expl-intro}
is well-suited to unveiling the role played by viscosity in the description of the energetic behavior of the system at jumps. Indeed, it
%\berin QUI COMMENTO SULLA ENERGY-DISSIPATION BALANCE CHE RIFLETTE LA COMPETIZIONE.. \erin
reflect the competition between the tendency of the system to be governed by \emph{viscous} dissipation both for the variable $z$ and for the variable  $u$
(with different rates if $\alpha \neq 1$),
 and its tendency to  be \emph{locally stable}  in $z$, and at equilibrium in $u$.
%\[
%\conjz {z}{q(t)}{-\rmD_z \calE(t,q(t))} = 0 \quad \text{i.e.} \quad -\rmD_z \calE(t,q(t)) \in \stab  {q(t)} \qquad \foraa\, t \in (0,T)
%\]
%for $z$, and the stationary condition
%\[
%\moreau {u}{q(t)}{-\rmD_u \calE(r,q(t))} =0 \quad \text{i.e.} \quad -\rmD_u \calE(t,q(t))=0 \qquad \foraa\, t \in (0,T)
%\]
for $u$, cf.\ also the discussion in Remark \ref{rmk:switch}.

Secondly,
to develop the analysis as $\eps \down 0$ for a family of
curves
$(u_\eps,z_\eps)_\eps \subset H^1 (0,T; \R^n \times \R^m)$ fulfilling \eqref{enid-eps-expl-intro}
  we will adopt a by now well-established technique from
\cite{ef-mie06}. Namely,
to capture the viscous transition paths  at jump points,
 we will   reparameterize the curves $(u_\eps,z_\eps)$, for instance by their arc-length. Hence we will address the analysis as
 $\eps \down 0$ of the  \emph{parameterized curves}
 $(\sft_\eps,\sfu_\eps,\sfz_\eps)_\eps $ defined on the interval $[0,S]$ with values  in the extended phase
 space $[0,T]\times\R^n \times \R^m$, with $\sft_\eps$ the rescaling functions and $\sfu_\eps:= u_\eps \circ \sft_\eps$, $\sfz_\eps:= z_\eps \circ \sft_\eps$.
 %With a suitable choice of the reparameterization,
 Under suitable conditions
  it can be proved that, up to a subsequence the curves $(\sft_\eps,\sfu_\eps,\sfz_\eps)_\eps $
 converge to a triple $(\sft,\sfu,\sfz) \in \AC ([0,S]; [0,T]\times \R^n \times \R^m)$. Its evolution is described by an energy-dissipation  balance %which is
 obtained
 by passing to the limit in the reparameterized version of \eqref{enid-eps-expl-intro}. cf.\ Theorem \ref{th:main}.
 We will refer to $(\sft,\sfu,\sfz)$ as a \emph{parameterized Balanced Viscosity} solution to the rate-independent system
 $(\R^n \times \R^m, \calE, \calR_0 + \eps \vpotname z + \eps^\alpha \vpotname u) $.

The main result of this paper, Theorem \ref{prop:diff-incl}, provides a more transparent reformulation of the energy-dissipation
balance defining  a parameterized Balanced Viscosity solution $(\sft,\sfu,\sfz)$.
It is
 in terms of a system of subdifferential inclusions fulfilled
by  the curve
$(\sft,\sfu,\sfz)$, namely
\begin{equation}
\label{diff-syst-intro}
\begin{aligned}
&
\thn u(s) \sfu'(s) + (1-\thn{u} (s)) \rmD_u \enet{\sft(s)}{\sfu(s)}{\sfz(s)} \ni 0 && \foraa\, s \in (0,S),
\\
&
(1-\thn{z} (s))  \partial\calR_0 (\sfq(s),\sfz'(s)) +
\thn z(s) \sfz'(s) +  (1-\thn{z} (s))  \rmD_z \enet{\sft(s)}{\sfu(s)}{\sfz(s)} \ni 0 && \foraa\, s \in (0,S),
\end{aligned}
\end{equation}
where the Borel functions $\thn u,\, \thn z : [0,S] \to [0,1]$ fulfill
\begin{equation}
\label{switching-intro}
\sft'(s)  \thn u(s) =  \sft'(s)  \thn z (s) =0 \qquad \foraa\, s \in (0,S),
\end{equation}
The latter condition reveals that the viscous terms $\sfu'(s) $ and $\sfz'(s) $ may contribute to \eqref{diff-syst-intro}
only at jumps of the system, corresponding to $\sft'(s)=0$ as the function $\sft$ records the (slow) external time scale.
 In this respect,
\eqref{diff-syst-intro}--\eqref{switching-intro} is akin to the (parameterized) subdifferential inclusion
\begin{equation}
\label{diff-single-intro}
\begin{aligned}
&
 \rmD_u \enet{\sft(s)}{\sfu(s)}{\sfz(s)} \ni 0 && \foraa\, s \in (0,S),
\\
&
 \partial\calR_0 (\sfz'(s)) +
\theta (s) \sfz'(s) +    \rmD_z \enet{\sft(s)}{\sfu(s)}{\sfz(s)} \ni 0 && \foraa\, s \in (0,S),
\end{aligned}
\end{equation}
with the Borel function $\theta: [0,S] \to [0,\infty)$
fulfilling
\begin{equation}
\label{switching-intro-2}
\sft'(s)  \theta (s) = 0 \qquad \foraa\, s \in (0,S).
\end{equation}
Indeed, \eqref{diff-single-intro} is the
 subdifferential reformulation for the parameterized Balanced Viscosity solutions
obtained by taking the limit as $\eps \down 0$ in \eqref{van-visco-intro}, where viscosity is added only to the flow rule.
However, note that \eqref{diff-syst-intro} has a much more complex structure than \eqref{diff-single-intro}. In addition to the switching condition
\eqref{switching-intro}, the functions $\thn u$ and $\thn z$ fulfill additional constraints, cf.\ Theorem  \ref{prop:diff-incl}.
They
differ in the three cases $\alpha>1$, $\alpha=1$, and $\alpha \in (0,1)$
 and show that viscosity in $u$ and $z$ pops back into the description of the system behavior at jumps, in a way depending on whether
 $u$ relaxes faster to equilibrium than $z$,
 $u$ and $z$ have the same relaxation rate, or $z$
relaxes faster to local stability than $u$.
%%%
%\beric Here, a description of our results: reparameterization, vanishing viscosity, parameterized energy balance, comments on its physical interpretation.. \eric
\paragraph{\bf Plan of the paper} In Section \ref{ss:2.1} we set up all the basic assumptions on the dissipation potentials $\calR_0$,
$\vpotname u$, and $\vpotname z$. Section \ref{s:2} is devoted to the generalized gradient system
driven by $\calE$ and the ``viscous'' potential $\calR_\eps := \calR_0+ \eps \vpotname z + \eps^\alpha \vpotname u$. In particular, we establish a series of estimates on the viscous solutions $(u_\eps,z_\eps)$ which will be at the core of the vanishing viscosity analysis, developed  in Section \ref{s:3} with Theorem
\ref{th:main}. In Section \ref{s:4} we will prove Theorem \ref{prop:diff-incl} and explore the mechanical interpretation of parameterized Balanced Viscosity solutions. Finally,  in Section \ref{s:5} we will illustrate this solution notion, focusing on how it varies in the cases $\alpha>1$, $\alpha=1$,
$\alpha \in (0,1)$,  in two different examples.
%%%
\paragraph{\bf Notation}
In what follows, we will denote by
$\langle \cdot, \cdot \rangle$ and by
 $|\cdot|$  the scalar product and the  norm in  any Euclidean space $\R^d$, with $d=n,\, m,\, n+m, \, \ldots$.
Moreover, we will use the same symbol $C$ to denote a positive constant depending on data, and possibly varying from line to line.
%%%%
\section{Setup}
\label{ss:2.1}
As mentioned in the introduction,
 we are going to address a more general version of system \eqref{eps-system-intro},
where the
$1$-positively homogeneous dissipation potential $\calR_0$,  as well as the
quadratic  potentials
$\vpotname u$ and $\vpotname z$
 for $u'$ and $z'$, are also depending on the state
variable
\[
q:= (u,z)\in \mathcal{Q}:= \R^{n} \times \R^{m}.
\]
Hence, the rate-independent system  is
\begin{equation}
\label{rip-syst}
\partial_{q'} \calR_0(q(t),z'(t)) +\rmD_q \calE(t,q(t)) \ni 0 \qquad \text{in } (0,T),
\end{equation}
namely
\begin{subequations}
\label{rip-limit}
\begin{align}
\label{rip-limit-1}
&
 \rmD_u \enet{t}{u(t)}{z(t)} =0   && \foraa\, t \in (0,T),
 \\
 &
 \label{rip-limit-2}
 \partial\calR_0(q(t), z'(t)) + \rmD_z \enet{t}{u(t)}{z(t)}  \ni 0 &&  \foraa\, t \in (0,T).
 \end{align}
\end{subequations}
We
approximate it with
%More precisely, we approximate (a generalized version of) the rate-independent system \eqref{limit-system-intro} with
the following
generalized gradient system
\begin{equation}
\label{gen-grad-syst}
\partial_{q'} \calR_\eps(q(t),q'(t)) +\rmD_q \calE(t,q(t)) \ni 0 \qquad \text{in } (0,T),
\end{equation}
where the overall dissipation potential $\calR_\eps$ is of the form
 \begin{equation}
\label{form-calR-eps}
\calR_\eps (q,q')= \calR_\eps (q,(u',z')):= \calR_0 (q,z') +\eps\vpot zq{z'} +\eps^\alpha\vpot uq{u'} \quad \text{with } \alpha >0.
\end{equation}

In what follows, let us specify our assumptions on the dissipation potentials $\calR_0$, $\vpotname z$, and $\vpotname u$.
\begin{description}
\item[\textbf{Dissipation}]  We require that
\begin{equation}
\label{ass:dissip-pot-R}
\tag{$\mathrm{{R}_0}$}
\begin{aligned}
&
\calR_0 \in \rmC^0  (\calQ \times \R^m ),  \quad \forall\, q \in \calQ \  \calR_0 (q,\cdot) \text{ is convex and $1$-positively homogeneous, and }
\\ &  \exists\, C_{0,R}, \, C_{1,R}>0 \ \forall\, (q,z')\in \calQ \times \R^m \, : \qquad
C_{0,R} |z'|\leq \calR_0 (q,z') \leq C_{1,R}|z'|,
\end{aligned}
\end{equation}
\begin{equation}
\label{ass:dissip-pot-Vz}
\tag{$\mathrm{{V}_z}$}
\begin{gathered}
\vpotname z : \calQ \times \R^m \to [0,\infty) \text{ is of the form }
\vpot zq{z'} = \frac12 \langle \vcof zq{z'}, z' \rangle \quad \text{with }
\\
\vcofname z  \in \rmC^0 (\calQ;\R^{m \times m}) \quad\text{and}\quad \exists\, C_{0,V}, \, C_{1,V}>0 \ \forall\, q\in \calQ \, : \qquad
C_{0,V} |z'|^2 \leq \vpot zq{z'} \leq C_{1,V}|z'|^2,
\end{gathered}
\end{equation}
\begin{equation}
\label{ass:dissip-pot-Vu}
\tag{$\mathrm{{V}_u}$}
\begin{gathered}
\vpotname u : \calQ \times \R^n \to [0,\infty) \text{ is of the form }
\vpot uq{u'} = \frac12 \langle \vcof uq{u'}, u' \rangle \quad \text{with }
\\
\vcofname u \in \rmC^0 (\calQ;\R^{n\times n}) \quad\text{and}\quad \exists\, \widetilde{C}_{0,V}, \, \widetilde{C}_{1,V}>0 \ \forall\, q\in \calQ \, : \qquad
\widetilde{C}_{0,V}|u'|^2 \leq \vpot uq{u'} \leq \widetilde{C}_{1,V}|u'|^2.
\end{gathered}
\end{equation}
\end{description}

For later use, let us recall that, due to the
$1$-homogeneity of $\calR_0(q,\cdot)$,
for every $q\in \calQ$ the convex analysis subdifferential $\partial \calR_0(q,\cdot) : \R^m \rightrightarrows \R^m$
is characterized by
\begin{equation}
\label{charact-1-homog}
\zeta \in \partial \calR_0(q,z') \quad \text{if and only if}\quad \begin{cases}
\langle \zeta, w\rangle \leq \calR_0(q,w) & \text{for all } w \in \R^m,
\\
\langle \zeta, z'\rangle \geq \calR_0(q,z')\,.
\end{cases}
\end{equation}
Furthermore,
observe that \eqref{ass:dissip-pot-Vz} and \eqref{ass:dissip-pot-Vu} ensure that for every $q \in \calQ$ the matrices
$ \vcof zq{}\in \R^{n \times n}$ and $ \vcof uq{}\in \R^{m \times m}$  are positive definite, uniformly with respect to $q$. Furthermore, for later use we observe that the conjugate
\[
\moreau uq{\eta} = \sup_{v \in \R^n} \left( \langle \eta, v \rangle - \vpot uqv \rangle\right) = \frac12 \langle \vcofinv uq{\eta}, \eta \rangle
\]
fulfills
\begin{equation}
\label{inv-vcof-1}
\overline{C}_0 |\eta|^2 \leq \moreau uq{\eta} \leq \overline{C}_1 |\eta|^2
\end{equation}
for some  $\overline{C}_0,\,\overline{C}_1>0$. We have the analogous coercivity  and growth properties for  $\calV_{\mathsf{z}}^*$.

Our assumptions concerning the energy functional $\calE$, expounded below,
 are typical of the \emph{variational approach} to gradient flows and generalized gradient systems. Since we
are in a finite-dimensional setting, to impose   \emph{coercivity} it is sufficient to
ask for boundedness of energy sublevels. The power-control condition will allow us to bound $\partial_t \calE$ in the derivation of the basic energy estimate
on system \eqref{gen-grad-syst}, cf.\ Lemma \ref{lemma:2.1} later on.
The smoothness of $\calE$ guarantees the validity of two further, key properties, i.e.\ the continuity of $\rmD_q\calE$, and the chain rule
(cf.\ \eqref{chain-rule} below), which will play a crucial role for our analysis.

Later on, in Section  \ref{s:2},  we will impose that $\calE$ is uniformly convex with respect to $u$. As we will see, this condition will be at the core of the proof of an estimate
for
 $\|u'\|_{L^1(0,T;\R^n)}$,  uniform with respect to
the parameter $\eps$. Observe that, unlike for $z'$ such estimate does not follow from the basic energy estimate on system  \eqref{gen-grad-syst}, since the overall dissipation potential
$\calR_\eps$ is degenerate in $u'$ as $\eps \down 0$. It will require additional careful calculations.
\begin{description}
\item[\textbf{Energy}] we assume that $\calE \in \rmC^1 ([0,T]\times \calQ)$   and that it is bounded from below by a positive constant
(indeed by adding a constant we can always reduce to this case). Furthermore, we require that
\begin{equation}
\label{ass:E}
\tag{$\mathrm{E}$}
\begin{aligned}
 &
 \exists\, C_{0,E}\,,  \widetilde{C}_{0,E}>0 \ \forall\, (t,q) \in [0,T]\times \calQ \,: \quad && \ene tq \geq C_{0,E}|q|^2 - \widetilde{C}_{0,E} && \text{\textbf{(coercivity),}}
 \\
 &
 \exists\,  C_{1,E}>0  \ \forall\, (t,q) \in [0,T]\times \calQ\, : \quad
&& |\partial_t \ene tq| \leq C_{1,E} \ene tq && \text{\textbf{(power control).}}
\end{aligned}
\end{equation}
\end{description}

In view of \eqref{form-calR-eps}, \eqref{ass:dissip-pot-Vz}, and \eqref{ass:dissip-pot-Vu},
the generalized gradient system
 \eqref{gen-grad-syst} reads
\begin{subequations}
\label{eps-system}
\begin{align}
\label{eq-u}
&
\eps^\alpha \vcof{u}{q(t)}{u'(t)} + \rmD_u \enet{t}{u(t)}{z(t)} =0 && \text{in } (0,T),
\\
&
\label{eq-z}
\eps \vcof{z}{q(t)}{z'(t)} +\partial\calR_0(z'(t)) + \rmD_z \enet{t}{u(t)}{z(t)} =0 && \text{in } (0,T).
\end{align}
\end{subequations}
\paragraph{\bf Existence of solutions to the generalized gradient system \eqref{gen-grad-syst}.}
It follows from the results in \cite{ColliVisintin90,MRS-dne} that, under the present assumptions,
for every $\eps>0$ there exists a solution $q_\eps \in H^{1}(0,T;\calQ)$ to the Cauchy problem for
\eqref{gen-grad-syst}.
% It was observed in \cite{MRS-dne} that
Observe that
 $q_\eps$ also fulfills
 the energy-dissipation identity
\begin{equation}
\label{enid-eps}
\ene t{q_\eps(t)} +
\int_s^t
 \calR_\eps (q_\eps(r), q_\eps'(r)) + \calR_\eps^* (q_\eps(r),-\rmD_q \calE(r,q_\eps(r)))  \dd r
 = \ene s{q_\eps(s)} + \int_s^t \partial_t \ene r{q_\eps(r)} \dd r.
\end{equation}
In \eqref{enid-eps},  the dual dissipation potential
$\calR_\eps^* : \calQ \times \R^{n+m} \to \R$ is the Fenchel-Moreau conjugate  of $\calR_\eps$, i.e.
\begin{equation}
\label{calReps}
\calR_\eps^*(q,\xi):= \sup_{v \in \calQ} \left(\langle \xi, v \rangle - \calR_\eps (q,v) \right).
\end{equation}
 In fact, by the   Fenchel equivalence the differential inclusion \eqref{gen-grad-syst}
reformulates as
\[
\calR_\eps (q_\eps(t),q_\eps'(t)) +\calR_\eps^* (q_\eps(t), - \rmD_q \calE(t,q_\eps(t)) ) =
\langle  - \rmD_q \calE(t,q_\eps(t)), q_\eps'(t) \rangle \qquad \foraa\, t \in (0,T).
\]
Combining this with  the chain rule
\begin{equation}
\label{chain-rule}
\frac{\dd }{\dd t }\calE (t,q(t)) =\partial_t \calE(t,q(t)) + \langle \rmD_q \calE(t,q(t)), q'(t) \rangle   \qquad \foraa\, t \in (0,T)
\end{equation}
along any curve $q\in \mathrm{AC}([0,T]; \calQ)$ and integrating in time,  we conclude \eqref{enid-eps}.

The energy balance \eqref{enid-eps}   will play a crucial role in our analysis: indeed, after deriving in Sec.\ \ref{s:2}  a series of a priori estimates, uniform with respect to the parameter $\eps>0$,
 we shall pass to the limit in the parameterized version of
\eqref{enid-eps} as $\eps\down 0$. We will  thus obtain  a (parameterized) energy-dissipation identity which encodes information on the behavior of the
limit
 system for $\eps=0$, in particular
  at the  jumps of the limit curve $q$ of the solutions $q_\eps$ to \eqref{gen-grad-syst}.
%%%
%%%

\section{A priori estimates}
\label{s:2}
%\subsection{A priori estimates uniform with respect to $\eps$}
%\label{ss:2.2}
In this section, we consider a family $(q_\eps)_\eps \subset H^1 (0,T;\calQ)$ of solutions to
the Cauchy problem for \eqref{gen-grad-syst},
with a converging   sequence of  initial data $(q_\eps^0)_\eps$, i.e.
\begin{equation}
\label{bded-data}
q_\eps^0 \to q^0
\end{equation}
for some $q^0 \in \calQ$.

 Our first result, Lemma \ref{lemma:2.1},
 provides a series of basic estimates on the functions $(q_\eps)$, as well as  a bound for  $\| z_\eps'\|_{L^1 (0,T;\R^m)}$, uniform with respect to $\eps$.
 It holds under conditions \eqref{ass:dissip-pot-R},
\eqref{ass:dissip-pot-Vz},
 \eqref{ass:dissip-pot-Vu}, \eqref{ass:E}, as well as \eqref{bded-data}.

 Under a further property of the dissipation potential
 $\vpotname u$ (cf.\ \eqref{need-esti} below),
 assuming \emph{uniform convexity} of $\calE$ with respect to the variable $u$, and  requiring an additional condition
the initial data $(q_\eps^0)_\eps$ (see \eqref{well-prep}),  in Proposition \ref{prop:aprio-eps} we
 will derive  the following crucial estimate,
%$(q_\eps') = (u_\eps',z_\eps')$,
uniform with respect to $\eps$:
\begin{equation}
\label{L1-est}
%\tag{$\mathrm{Est}$}
\|q_\eps'\|_{L^1 (0,T;\R^{n+m})} \leq C.
\end{equation}

We start with the following result, which does not require the above
mentioned enhanced conditions.

\begin{lemma}
\label{lemma:2.1}
Let $\alpha>0$.  Assume \eqref{ass:dissip-pot-R},
\eqref{ass:dissip-pot-Vz}, \eqref{ass:dissip-pot-Vu}, \eqref{ass:E},
and \eqref{bded-data}.  Then, there exists a constant $C>0$ such that
for every $\eps>0$
\begin{subequations}
\begin{align}
 \label{bound-energies}
 &
\text{(a)} \quad  \sup_{t \in [0,T]} \ene t{q_\eps(t)}\leq C,
  \\
  &
  \label{bound-q-eps}
  \text{(b)} \quad
 \sup_{t \in [0,T]}  |q_\eps(t)| \leq C,
  \\
  &
 \label{est-z}
 \text{(c)} \quad
 \int_0^T |z_\eps'(r)| \dd r \leq C.
 \end{align}
\end{subequations}
\end{lemma}
%%%
\begin{proof}
We  exploit the energy identity \eqref{enid-eps}.
Observe that $\calR_\eps^* (q,\xi) \geq 0$ for all $(q,\xi) \in\calQ \times  R^{n+m}$.
Therefore,  we deduce from  \eqref{enid-eps}
that
\[
\ene t{q_\eps(t)} \leq \ene 0{q_\eps (0)} + \int_0^t \partial_t \ene r{q_\eps(r)} \dd r
\leq C+C_{1,E}\int_0^t \ene r{q_\eps(r)} \dd r,
\]
where we have
used the power control from \eqref{ass:E}
and
 the fact that $\ene 0{q_\eps (0)}  \leq C$, since the $(q_\eps (0))_\eps$ is bounded. The Gronwall Lemma then yields \eqref{bound-energies},
and \eqref{bound-q-eps} ensues from the coercivity of $\calE$.
 Using again the power control,
we ultimately infer from \eqref{enid-eps} that
\begin{equation}
\label{en-dissip-est}
\int_0^T
 \calR_\eps (q_\eps(r), q_\eps'(r)) + \calR_\eps^* (q_\eps(r),-\rmD_q \calE(r,q_\eps(r)))  \dd r \leq C.
 \end{equation}
 In particular,
 $\int_0^T \calR_0 (q_\eps(r), z_\eps'(r)) \dd r \leq C$, whence \eqref{est-z}  by \eqref{ass:dissip-pot-R}.
 \end{proof}

The derivation of  the  $L^1(0,T;\R^n)$-estimate for $(u_\eps')_\eps$
similar to \eqref{est-z}
clearly does not follow from \eqref{enid-eps},
 which only yields $\int_0^T \eps^{\alpha} |u_\eps'(r)|^2 \dd r \leq C$ via \eqref{en-dissip-est}
and \eqref{ass:dissip-pot-Vu}.
 It is indeed more involved,
 and, as already mentioned, it  strongly relies on the uniform convexity of $\calE$ with respect to $u$. Furthermore,
  we are able to obtain it only under the simplifying condition that
the dissipation potential
 $\vpotname u$ in fact \emph{does not} depend on the state variable $q$,
 and under an additional well-preparedness condition on the data $(q_\eps^0)_\eps$,
 ensuring that the forces $\rmD_u \ene{0}{q_\eps^0}$ tend to zero, as $\eps \down 0$, with rate $\eps^\alpha$.
\begin{proposition}
\label{prop:aprio-eps}
Let $\alpha>0$.
Assume \eqref{ass:dissip-pot-R},
\eqref{ass:dissip-pot-Vz},
 \eqref{ass:dissip-pot-Vu}, and \eqref{ass:E}.
 In addition, suppose that  %$\vpotname{u}$ does not depend on  the state $q$, i.e.
 \begin{equation}
 \label{need-esti}
 \tag{$\mathrm{V}_{u,1}$}
 \rmD_q \vcofname{u} (q) = 0 \quad \text{for all } q \in \calQ,
 \end{equation}
%  that $\calE$ is uniformly convex w.r.t.\ $u$
  \begin{equation}
 \label{en-plus}
 \tag{$\mathrm{E}_{1}$}
 \begin{aligned}
&\calE \in \rmC^2 ([0,T]\times \calQ)
\quad \text{and}
\\
&\exists\, \mu>0  \ \forall\, (t,q) \in  [0,T]\times \calQ\, : \
\rmD_u^2 \ene tq\geq \mu \mathbb{I}_{\R^{n \times n}}  \quad
\text{\textbf{(uniform convexity w.r.t.\ $u$),}}
\end{aligned}
\end{equation}
 and that the initial data $(q_\eps^0)_\eps $  complying with  \eqref{bded-data}
 also fulfill
  %are also well-prepared, in the sense that
 \begin{equation}
 \label{well-prep}
 |\rmD_u \ene{0}{q_\eps^0}| \leq C \eps^{\alpha}.
 \end{equation}
 Then,
 there exists a constant $C>0$ such that for every $\eps>0$
 \begin{equation}
 \label{est-u}
 \|u_\eps'(t)\|_{L^1(0,T;\R^{n})} \leq C.
 \end{equation}
\end{proposition}
\begin{proof}
It follows from  \eqref{need-esti} that  there exists a given matrix $\overline{\bbV}_{\mathsf{u}} \in  \R^{n\times n}$
such that
\begin{equation}
\label{yes-later}
\vcof uq{} \equiv \overline{\mathbb{V}}_{\mathsf{u}} \quad \text{for all $q\in\calQ$,}
\end{equation}
 so that
\begin{equation}
\label{v-const}
\vpot uq{u'} = \calV_{\mathsf{u}} (u') :=
\frac12 \langle   \overline{\mathbb{V}}_{\mathsf{u}} u',u' \rangle.
\end{equation}
  Therefore \eqref{eq-u} reduces to
\begin{equation}
\label{simpler}
\eps^\alpha \overline{\mathbb{V}}_{\mathsf{u}} u_\eps'(t) + \rmD_u
\enet{t}{u_\eps(t)}{z_\eps(t)} =0
\qquad \foraa\, t \in (0,T).
\end{equation}
We differentiate \eqref{simpler} in time,
 %\footnote{\beric  this is a formal calculation, we should make it
 %rigorous either with difference quotients, or
 %on the time-discrete level... shall we do so, or shall we just say
 %that it can be made rigorous?\eric}
and test the resulting equation by $u_\eps'$.  Thus we obtain for
almost all $t\in (0,T)$
\begin{equation}
\label{test-nu}
\begin{aligned}
0 & = \eps^\alpha \langle  \overline{\mathbb{V}}_{\mathsf{u}}
     u_\eps{''}(t), u_\eps'(t) \rangle
+ \langle \rmD_u^2 \enet{t}{u_\eps(t)}{z_\eps(t)}  [u_\eps'(t)], u_\eps'(t) \rangle
+ \langle \rmD_{u,z}^2  \enet{t}{u_\eps(t)}{z_\eps(t)} [ u_\eps'(t)], z_\eps'(t) \rangle
\\
 &
  \doteq S_1+S_2+S_3,
  \end{aligned}
\end{equation}
%\beric notation?? \eric
where $ \rmD_{u,z}^2 $ denotes the second-order mixed derivative.
Observe that
\begin{align*}
&
S_1 = \frac{\eps^\alpha}2 \frac{\dd}{\dd t } \calV_{\mathsf{u}} (u_\eps'),
&&
S_2\geq \mu |u_\eps'^2| \geq \tilde{\mu}  \calV_{\mathsf{u}} (u_\eps'),
\\
&
S_3 \geq - C |u_\eps'||z_\eps'|  \geq -C \sqrt{\calV_{\mathsf{u}} (u_\eps')} |z_\eps'|.
\end{align*}
Indeed, to estimate $S_2$ we have used the uniform convexity of
$\calE(t,\cdot,z)$, and the growth of $\vpotname u$ from
\eqref{ass:dissip-pot-Vu}. The estimate for $S_3$ follows from
$\sup_{t\in (0,T)} | \rmD_{u,z}^2 \enet{t}{u_\eps(t)}{z_\eps(t)}| \leq
C$, due to \eqref{bound-q-eps} and the fact that $\rmD_{u,z}^2 \calE $
is continuous on $ [0,T] \times \calQ$, and again from
\eqref{ass:dissip-pot-Vu}.  We thus infer from \eqref{test-nu} that
\[
 \frac{\dd }{\dd t}   \calV_{\mathsf{u}} (u_\eps'(t)) + \frac{\tilde{\mu}}{\eps^\alpha} \calV_{\mathsf{u}} (u_\eps'(t)) \leq
   \frac{C}{\eps^\alpha} \sqrt{\calV_{\mathsf{u}} (u_\eps'(t))} |z_\eps'(t)| \qquad \foraa\, t \in (0,T),
\]
 which rephrases as
\[
\nu_\eps(t) \nu_\eps'(t) + \frac{\tilde{\mu}}{\eps^\alpha} \nu_\eps^2(t) \leq  \frac{C}{\eps^\alpha}\nu_\eps(t)  |z_\eps'(t)|
\]
where we have used the place-holder $\nu_\eps(t) :=
\sqrt{\calV_{\mathsf{u}} (u_\eps'(t))}$.  We now argue as in
\cite{Miel08?DEMF} and observe that, without loss of generality, we
may suppose that $\nu_\eps(t) >0$ (otherwise, we replace it by
$\tilde{\nu}_\eps = \sqrt{\nu_\eps +\delta}$, which satisfies the same
estimate, and then let $\delta \down 0$), Hence, we deduce
\[
 \nu_\eps'(t) + \frac{\tilde{\mu}}{\eps^\alpha} \nu_\eps(t) \leq
 \frac{C}{\eps^\alpha}|z_\eps'(t)|.
\]
Applying the Gronwall lemma %from, e.g.\ \cite{GMPZ},
we obtain
\begin{equation}
 \label{integrate-time}
 \nu_\eps(t) \leq C \exp\left(-\frac{\tilde \mu}{\eps^\alpha}t
 \right) \nu_\eps(0)  + \frac{C}{\eps^\alpha}  \int_0^t \exp
 \left(-\frac{\tilde \mu}{\eps^\alpha} (t-r)\right) |z_\eps'(r)| \dd
 r  \doteq a_1^{\eps}(t) +  a_2^{\eps}(t)
\end{equation}
for all $t\in (0,T)$.  We integrate the above estimate on $(0,T)$.
Now, observe that \eqref{well-prep} guarantees that $\nu_\eps(0)=
\sqrt{\calV_{\mathsf{u}} (u_\eps'(0))} \leq C
|\overline{V}_{\mathsf{u}} u_\eps'(0)| = C \eps^{-\alpha} |\rmD\calE
(0,u_\eps(0))| \leq C $.  Hence, we find $\|a_1^{\eps}\|_{L^1(0,T)}
\leq C \nu_\eps (0) \leq C_1$.
%\begin{equation}
%\label{needed}
%   \sup_\eps \exp\left(-\frac{\tilde \mu}{\eps^\alpha} \right )  \nu_\eps(0) \leq C.
%\end{equation}
%since, by \eqref{eq-u}, $ \nu_\eps(0)  =
 %\vcof{u}{q_\eps(0)}{u_\eps'(0)} = -\eps^{-\alpha} \rmD_q \ene 0{q_\eps(0)}$.
In order to estimate $a_2^\eps$ we use the Young inequality for
convolutions, which yields
\[
 \|a_2^{\eps}\|_{L^1(0,T)} =  \frac{C}{\eps^\alpha} \int_0^T   \int_0^t
 \exp \left(-\frac{\tilde \mu}{4\eps^\alpha} (t-r) \right) |z_\eps'(r)|
 \dd  r \dd t  \leq \frac{C}{\eps^\alpha} \left( \int_0^T  \exp
  \left(-\frac{\tilde \mu}{\eps^\alpha} t \right) \dd t  \right)
 \left( \int_0^T |z_\eps'(t)| \dd t \right)  \leq C_2
\]
where we have exploited the a priori estimate \eqref{est-z} for
$z_\eps'$. Thus,  \eqref{integrate-time} implies \eqref{est-u},
and we are done.
 \end{proof}

\section{Limit passage with vanishing viscosity}
\label{s:3}
In this section, we assume that we are given a sequence $(q_\eps)_\eps
\subset H^1(0,T;\calQ)$ of solutions to \eqref{gen-grad-syst},
satisfying the initial conditions $q_\eps(0) = q_\eps^0$, such that
estimate \eqref{L1-est} holds.  As we have shown in Proposition
\ref{prop:aprio-eps}, the well-preparedness \eqref{well-prep} of the
initial data $(q_\eps^0)_\eps$, the condition that the dissipation
potential $\vpotname u$ does not depend on the state $q$, and the
uniform convexity \eqref{en-plus} of $\calE$ with respect to $u$
guarantee the validity of \eqref{L1-est}. However, these conditions
are not needed for the vanishing viscosity analysis.
%\footnote{\beric  A question: it seems to me that the uniform convexity of
%$\calE(t,\cdot,z)$ is neither explicitly used in the limit passage as $\eps \down 0$, is it?
%\eric}
Therefore, hereafter
we will no longer impose% the well-preparedness
  \eqref{well-prep}, we will allow for a state-dependent  dissipation potential $\vpotname u = \vpot uq{u'}$,
  and we will stay with the basic conditions \eqref{ass:E} on $\calE$.
%although we do need conditions \eqref{need-esti} and \eqref{well-prep} in order to actually prove estimate \eqref{L1-est}.

\paragraph{\bf  The energy-dissipation balance.}
%\label{ss:3.1}
Following the variational approach of \cite{MRS09,MRS10,mielke-rossi-savare2013}, we will pass to the limit in (a
\emph{parameterized} version of) the energy identity \eqref{enid-eps}.

Preliminarily, let us
explicitly calculate the convex-conjugate of the dissipation potential $\calR_\eps$ \eqref{form-calR-eps}.
\begin{lemma}
\label{l:cvx-conj-Reps}
Assume \eqref{ass:dissip-pot-R},  \eqref{ass:dissip-pot-Vz},  and \eqref{ass:dissip-pot-Vu}. Then, the Fenchel-Moreau conjugate  \eqref{calReps} of
$\calR_\eps$ is given by
\begin{equation}
\label{to-refer-to}
\calR_\eps^* (q,\xi)=  \frac1{\eps} \conjz {z}{q}{\zeta}  + \frac1{\eps^\alpha} \moreau u{q}{\eta} \qquad \text{for all } q \in \calQ  \text{ and  } \xi = (\eta,\zeta) \in \R^{n+m},
\end{equation}
where $\moreau u{q}{\cdot}$ is the conjugate of $\vpot u{q}{\cdot}$, and
\begin{equation}
\label{conju-z}
\conjz {z}{q}{\zeta} = \min_{\omega \in \stab  q} \moreau z{q}{\zeta-\omega} \qquad \text{with } \stab q:=   \partial\calR_0 (q,0),
\end{equation}
 $\moreau z{q}{\cdot}$  is  the conjugate of $\vpot z{q}{\cdot}$, while
 $\calW_{\mathsf{z}}^*$
 is the conjugate of $\calR_0 + \calV_{\mathsf{z}}$.
  % and
 %for $q$ fixed,
 %\begin{equation}
%\label{z-stable}
%\stab q \text{ is the \emph{stable set} }
%\end{equation}
\end{lemma}
\begin{proof}
Since $\calR_\eps(q,\cdot)$ is given by the sum of a contribution in the sole variable $z'$ and another in the sole
variable $u'$, we have
\[
\calR_\eps^* (q,\xi)=
(\eps^\alpha \vpotname{u})^*(q,\eta) + \conjzspe {z}{\eps}{q}{\zeta} \qquad
 \text{for all $\xi = (\eta,\zeta) \in \R^{n + m}$} \qquad
\]
where we
have used the place-holder
$ %\begin{equation} \label{def-conjz}
\conjzspe {z}{\eps}{q}{\zeta} := \left(\calR_0 (q,\cdot) + \eps \vpot z{q}{\cdot}\right)^* (\zeta).
$ % \end{equation}
Now, taking into account that $\vpotname u$ is quadratic, there holds
\[
(\eps^\alpha \vpotname{u})^*(q,\eta) =
\eps^\alpha \vpotname{u}^* \left(q, \frac1{\eps^\alpha} \eta \right) =
 \frac1{\eps^\alpha} \moreau u{q}{\eta},
\]
whereas the $\inf$-$\sup$ convolution formula (see e.g.\
\cite{IofTih}) yields $\conjzspe {z}{\eps}{q}{\zeta} = \frac1{\eps}
\conjz {z}{q}{\zeta}$ with $\conjz {z}{q}{\cdot}$ from
\eqref{conju-z}.
\end{proof}

In view of \eqref{to-refer-to}, the energy identity  \eqref{enid-eps} rewrites as
\begin{equation}
\label{enid-eps-expl}
\begin{aligned}
&
\ene t{q_\eps(t)}   +
\int_s^t
 \calR_0 (q_\eps(r), z_\eps'(r)) +
 \eps \vpot{z}{q_\eps(r)}{z_\eps'(r)} +\eps^\alpha \vpot{u}{q_\eps(r)}{u_\eps'(r)} \dd r
 \\ &\quad  +
\int_s^t
 \frac1{\eps}\conjz {z}{q_\eps(r)}{-\rmD_z \calE(r,q_\eps(r))}
 +\frac{1}{\eps^\alpha}\moreau {u}{q_\eps(r)}{-\rmD_u \calE(r,q_\eps(r))}
 \dd r
 \\
 &
 = \ene s{q_\eps(s)} + \int_s^t \partial_t \ene r{q_\eps(r)} \dd r.
 \end{aligned}
\end{equation}
%\berin QUI COMMENTO SULLA ENERGY-DISSIPATION BALANCE CHE RIFLETTE LA COMPETIZIONE.. \erin
In fact,  the second and the third integral terms on the left-hand side of \eqref{enid-eps-expl} reflect the competition between the tendency of the system to be governed by \emph{viscous} dissipation both for the variable $z$ and for the variable  $u$, and its tendency to  fulfill the \emph{local stability} condition
\[
\conjz {z}{q(t)}{-\rmD_z \calE(t,q(t))} = 0 \quad \text{i.e.} \quad -\rmD_z \calE(t,q(t)) \in \stab  {q(t)} \qquad \foraa\, t \in (0,T)
\]
for $z$, and the equilibrium condition
\[
\moreau {u}{q(t)}{-\rmD_u \calE(r,q(t))} =0 \quad \text{i.e.} \quad -\rmD_u \calE(t,q(t))=0 \qquad \foraa\, t \in (0,T)
\]
for $u$, cf.\ also the discussion in Remark \ref{rmk:switch}.
\paragraph{\bf The parameterized energy-dissipation balance.}
%\label{ss:3.2}
We now consider the parameterized curves
$(\sft_\eps,\sfq_\eps) : [0,S_\eps] \to  [0,T] \times \calQ$, where  for every $\eps>0$ the rescaling function
$\sft_\eps: [0,S_\eps] \to  [0,T]$ is strictly increasing, and
%\begin{equation}
%\label{link}
$\sfq_\eps(s) = q_\eps(\sft_\eps(s)).$
%\end{equation}
We shall suppose that $\sup_{\eps>0} S_\eps<\infty$, and that
\begin{equation}
\label{normalization}
\exists\, C>0 \quad \forall\, \eps>0 \quad \forall\, s \in [0,S_\eps]\, : \qquad
    \sft_\eps'(s) + |\sfq_\eps'(s)| \leq C.
\end{equation}
\begin{remark}
\upshape
\label{rmk:arclength}
For instance, as in \cite{ef-mie06, MRS09} we might
choose
\begin{equation}
\label{arclength-choice}
\sft_\eps:= \sigma_\eps^{-1} \quad \text{ with } \sigma_\eps(t):= \int_0^t \left( 1+|q_\eps'(r)| \right) \dd r,
\end{equation}
and set $S_\eps:= \sigma_\eps (T)$. In fact,
estimate \eqref{L1-est} ensures that $\sup_\eps S_\eps <\infty$. With the choice
\eqref{arclength-choice}
for $\sft_\eps$, the functions $(\sft_\eps,\sfq_\eps)$
fulfill the normalization condition
\[
\sft_\eps'(s) + |\sfq_\eps'(s)| =1 \qquad \text{for almost all } s\in (0,S_\eps).
\]
%In fact, as in  \cite{MRS10,mielke-rossi-savare2011} we might also rescale
%the functions $q_\eps$ by their energy-dissipation arclength, defined by
%\begin{equation}
%\label{energy-dissipation-arclength}
%\begin{aligned}
%\tilde{\sigma}_\eps (t): & =  \int_0^t \left( 1+\Me{\eps}{q_\eps(r)}{\tau}{q_\eps'(r)}{-\rmD_q \ene r{q_\eps(r)}}
%\right) \dd r
%\\
%&
%\int_0^t \left( 1+ \calR_0 (q_\eps(r),z\eps'(r)) +\eps\vpot z{q_\eps(r)}{z_\eps'(r)} +
%\eps^\alpha \vpot u{q_\eps(r)}{u_\eps'(r)} +
%\frac1\eps \conjz {z}{q_\eps(r)}{-\rmD_z \ene r{q_\eps(r)}} + \frac1{\eps^\alpha} \moreau u{q_\eps(r)}{-\rmD_u \ene r{q_\eps(r)}}
%\right) \dd r\,.
%\end{aligned}
%\end{equation}
%Indeed,
%the very same arguments as in the proof of Lemma \ref{lemma:2.1} allow us to conclude from  the energy identity \eqref{enid-eps-expl} that
%  the sequence
%$(\tilde{S}_\eps  = \tilde{\sigma}_\eps (T) _\eps$ is bounded. The resulting rescaled sequences $\tilde{\mathsf{q}}_\eps $, with
%$\tilde{\mathsf{q}}_\eps(s) = q_\eps(\tilde{\mathsf{t}}_\eps(s))$ and $(\tilde{\mathsf{t}}_\eps = \tilde{\sigma}_\eps^{-1}$, still fulfill
%\eqref{normalization}, as estimate ?????
\end{remark}

For the  parameterized curves
$(\sft_\eps,\sfq_\eps)$, the energy-dissipation balance \eqref{enid-eps-expl} reads
\begin{equation}
\label{enid-eps-param}
\begin{aligned}
&
\ene {\sft_\eps(s_2)}{\sfq_\eps(s_2)} +\int_{s_1}^{s_2}
\Me{\eps}{\sfq_\eps(r)}{\sft_\eps'(r)}{\sfq'_\eps(r)}{-\rmD_q\ene{\sft_\eps(r)}{\sfq_\eps(r)}} \
\dd r
\\
& = \ene {\sft_\eps(s_1)}{\sfq_\eps(s_1)}
+ \int_{s_1}^{s_2} \partial_t \ene{\sft_\eps(r)}{\sfq_\eps(r)} \sft_\eps'(r) \dd  r \qquad \text{for all } 0 \leq s_1 \leq s_2 \leq S,
\end{aligned}
\end{equation}
where we have used the dissipation functional
\begin{equation}
\label{Me-def}
\begin{aligned}
\Me{\eps}{q}{\tau}{q'}{\xi}  &  =
\Me {\eps}{q}{\tau}{(u',z')}{(\eta,\zeta)}
\\ &
:=
\calR_0 (q,z') +\frac{\eps}{\tau} \vpot zq{z'} +
\frac{\eps^\alpha}{\tau} \vpot uq{u'} +
\frac\tau\eps \conjz {z}{q}{\zeta} + \frac\tau{\eps^\alpha} \moreau uq\eta.
\end{aligned}
\end{equation}
The passage from  \eqref{enid-eps-expl} to \eqref{enid-eps-param} follows from the change of variables
$t  \to \sft_\eps (r)$, whence $\dd t \to  \sft_\eps'(r) \dd r$, while $q_\eps'(t) \to  \frac{1}{ \sft_\eps'(r) } \sfq_\eps'(r)$.
In order to pass to the limit in \eqref{enid-eps-param} as $\eps\down0$, it is crucial to investigate the
$\Gamma$-convergence properties of the family of functionals $(\mathcal{M}_\eps)_\eps$. The following result reveals that
the $\Gamma$-limit of $(\mathcal{M}_\eps)_\eps$ depends on whether
the parameter $\alpha$ is above, equal, or below the threshold value $1$.
Let us point out that, for $\alpha\in (0,1)$, setting $\delta = \eps^{\alpha}$ we rewrite $\mathcal{M}_\eps$ as
\begin{equation}
\label{specular}
\Me {\eps}{q}{\tau}{(u',z')}{(\eta,\zeta)} =  \calR_0 (q,z')
+\frac{\delta^{1/\alpha}}{\tau} \vpot zq{z'}
+
\frac{\delta}{\tau} \vpot uq{u'}
+\frac\tau{\delta^{1/\alpha}} \conjz {z}{q}{\zeta}
 + \frac\tau{\delta} \moreau uq\eta
\end{equation}
with $1/\alpha >1$.
It is thus natural to expect that the upcoming results will be specular in the cases $\alpha\in (0,1)$ and $\alpha>1$.
\begin{proposition}
\label{prop-Gamma-conv}
%Let $\alpha>0$.
Assume \eqref{ass:dissip-pot-R}, \eqref{ass:dissip-pot-Vz}, \eqref{ass:dissip-pot-Vu}, and \eqref{ass:E}.
Then,
the functionals $(\mathcal{M}_\eps)_\eps$ $\Gamma$-converge as $\eps\down 0$ to
$\mathcal{M}_0: \calQ\times [0,\infty) \times  \calQ \times \R^{n+m} \to [0,\infty] $
defined by
\begin{equation}
\label{basic-one}
\Mo q\tau{(u',z')}{(\eta,\zeta)}:= \calR_0 (q,z') + \Mored q\tau{(u',z')}{(\eta,\zeta)},
\end{equation}
where for $\tau>0$ we have
\begin{equation}
\label{Mored-basic}
%\text{for } \tau>0 \, : \quad
\Mored q\tau{(u',z')}{(\eta,\zeta)} = \left\{ \begin{array}{ll}
0 & \text{if }  \conjz {z}{q}{\zeta} =\moreau uq{\eta} =0,
\\
\infty & \text{if } \conjz {z}{q}{\zeta} +\moreau uq{\eta} >0,
\end{array}
\right.
\end{equation}
while for $\tau=0$ we have the following cases:
\begin{itemize}
\item For $\alpha>1$
\begin{equation}
\label{def-Mop1}
\Mored q0{(u',z')}{(\eta,\zeta)} =
\left\{
\begin{array}{ll}
2 \sqrt{\vpot{u}{q}{u'}}\, \sqrt{\moreau uq{\eta}} & \text{ if }  \vpot{z}{q}{z'}=0,
\\
2\sqrt{\vpot{z}{q}{z'}}  \, \sqrt{\conjz {z}{q}{\zeta}} & \text{ if }   \moreau uq{\eta} =0,
\\
\infty & \text{ if } \vpot{z}{q}{z'}\, \moreau uq{\eta} >0,
\end{array}
\right.
\end{equation}
\item For $\alpha=1$
\begin{equation}
\label{def-Mou1}
\Mored q0{(u',z')}{(\eta,\zeta)} = 2 \sqrt{\vpot{z}{q}{z'} + \vpot{u}{q}{u'}}\,\sqrt{\conjz {z}{q}{\zeta} +\moreau uq{\eta} },
\end{equation}
\item For $\alpha\in (0,1)$
\begin{equation}
\label{def-Mom1}
\Mored q0{(u',z')}{(\eta,\zeta)} =
\left\{
\begin{array}{ll}
2 \sqrt{\vpot{u}{q}{u'}} \, \sqrt{\moreau uq{\eta}} & \text{ if }  \conjz {z}{q}{\zeta}=0,
\\
2 \sqrt{\vpot{z}{q}{z'}} \,  \sqrt{\conjz {z}{q}{\zeta}} & \text{ if }   \vpot{u}{q}{u'} =0,
\\
\infty & \text{ if } \vpot{u}{q}{u'}\, \conjz {z}{q}{\zeta} >0.
\end{array}
\right.
\end{equation}
\end{itemize}
Moreover, if $(\tau_\eps,q_\eps') \weakto (\tau,q') $ in $L^1 (0,S;
(0,T)\times \calQ)$ and if $(q_\eps,\xi_\eps)\to (q,\xi) $ in $L^1
(0,S; \calQ \times \R^{n+m})$,
%$ \xi_\eps \weakto \xi $ in $L^1 (0,S;  \calQ \times \R^{n+m})$,
 then for every $0\leq s_1\leq s_2\leq S$
\begin{equation}
\label{ioffe-refined}
\liminf_{\eps \down 0}
\int_0^S \Me{\eps}{q_\eps(s)}{\tau_\eps(s)}{q_\eps'(s)}{\xi_\eps(s)} \dd s
\geq \int_0^S \Mo{q(s)}{\tau(s)}{q'(s)}{\xi(s)} \dd s\,.
\end{equation}
\end{proposition}
\begin{remark}
\label{rmk:switch}
\upshape
Let us briefly comment on  the expression \eqref{basic-one} of the $\Gamma$-limit $\mathcal{M}_0$. To do so, we rephrase
the constraints arising in the switching conditions for the reduced functional $\moredname$, cf.\ \eqref{Mored-basic},  \eqref{def-Mop1}, and
\eqref{def-Mom1}.  Indeed, it follows from \eqref{ass:dissip-pot-Vz} and \eqref{ass:dissip-pot-Vu}  (cf.\ \eqref{inv-vcof-1}) that
\[
\begin{array}{lllllll}
 &  \vpot{z}{q}{z'} =0  & \Leftrightarrow & z'=0, \qquad   &  \vpot{u}{q}{u'} =0  & \Leftrightarrow & u'=0,
 \\
 &
 \moreau uq{\eta} = 0 & \Leftrightarrow & \eta=0, \qquad   &   \conjz {z}{q}{\zeta}=0 & \Leftrightarrow & \zeta \in  \stab q=  \partial\calR_0 (q,0).
\end{array}
\]
Therefore, from \eqref{Mored-basic} we read that for $\tau>0$ the functional $\Mored{q}{\tau}{\cdot}{\cdot}$ is finite (and indeed equal to $0$) only  for
$\eta$ and $\zeta$ fulfilling
\[
\eta =0, \qquad
 \zeta \in \stab  q \,.
\]

For $\tau=0$, in the case $\alpha>1$, $\Mored{q}{0}{\cdot}{\cdot}$ is finite if and only if either $z'=0$ or $\eta =0$. As we will see when discussing the physical interpretation of our %forthcoming
vanishing-viscosity result, this means that, at a jump (i.e.\ when $\tau=0$), either $z'=0$, i.e.\   $z$ is frozen, or $u$
fulfills the equilibrium condition $\eta =\rmD_u \ene tu=0$.

 Also in view of \eqref{specular}, the switching conditions for $\alpha\in (0,1)$ are specular
 to the ones for $\alpha>1$
 in a generalized sense.
 In fact, $\Mored{q}{0}{\cdot}{\cdot}$ is finite if and only if either
 %$u'=0$, i.e.\
 $u$ is frozen, or $\zeta = \rmD_z \ene tz \ \in  \stab q$,  meaning that  $z$ fulfills the \emph{local stability} condition.

%\beric here make more explicit the switching conditions in the above formula, for instance recalling that $ \conjz {z}{q}{\zeta}=0$ iff you are in the stable region... \eric
\end{remark}
\begin{proof}
Observe that
\[
\Me {\eps}{q}{\tau}{(u',z')}{(\eta,\zeta)}=
\calR_0 (q,z') +
\Mered q\tau{(u',z')}{(\eta,\zeta)}
\]
with $\Mered q\tau{(u',z')}{(\eta,\zeta)} : = \frac{\eps}{\tau} \vpot zq{z'} +
\frac{\eps^\alpha}{\tau} \vpot uq{u'} +
\frac\tau\eps \conjz {z}{q}{\zeta} + \frac\tau{\eps^\alpha} \moreau uq\eta$.
Since $\calR_0$ is continuous with respect to both variables $q$ and $z$ and does not depend on $\eps$, it is
clearly sufficient to prove that the functionals $\mathcal{M}_\eps^{\mathrm{red}}$
$\Gamma$-converge to $\mathcal{M}_0^{\mathrm{red}}$, namely
\begin{align}
&
\label{liminf}
\begin{aligned}
&\Gamma\text{-}\liminf \text{ estimate: }
\\
&
\quad (q_\eps, \tau_\eps,u_\eps',z_\eps',\eta_\eps,\zeta_\eps)  \to
(q,\tau,u',z',\eta,\zeta)  \ \ \text{for }\eps \to 0 \\
&\qquad \Longrightarrow \ \   \Mored q\tau{(u',z')}{(\eta,\zeta)}  \leq
\liminf_{\eps \downarrow 0} \Mered
{q_\eps}{\tau_\eps}{(u_\eps',z_\eps')}{(\eta_\eps,\zeta_\eps)}, &&
 \end{aligned}
\\
&
\label{limsup}
\begin{aligned}
&
\Gamma\text{-}\limsup \text{ estimate: }
\\
& \quad
\forall\,  (q,\tau,u',z',\eta,\zeta) \ \exists\,
(q_\eps, \tau_\eps,u_\eps',z_\eps',\eta_\eps,\zeta_\eps)_\eps \, : \
\\
&\qquad \qquad
\begin{cases}
(q_\eps,\tau_\eps,u_\eps',z_\eps',\eta_\eps,\zeta_\eps) \to (q,\tau,u',z',\eta,\zeta) \qquad \text{and}
\\
 \Mored q\tau{(u',z')}{(\eta,\zeta)}  \geq \limsup_{\eps \downarrow 0} \Mered {q_\eps}{\tau_\eps}{(u_\eps',z_\eps')}{(\eta_\eps,\zeta_\eps)}.
 \end{cases}
 \end{aligned}
\end{align}
Preliminarily, observe that minimizing with respect to $\tau$ we
obtain the lower bound
\begin{equation}
\label{crucial-observation}
\Mered q\tau{(u',z')}{(\eta,\zeta)}
\geq 2 \sqrt{\eps\vpot zq{z'} +
\eps^\alpha \vpot uq{u'} } \sqrt{
\frac1\eps \conjz {z}{q}{\zeta} + \frac1{\eps^\alpha} \moreau uq\eta}.
\end{equation}
%\end{proof}
%\erin
%\end{document}

In all the three cases $\alpha>1$, $\alpha=1$, and $\alpha\in (0,1)$,  the expression \eqref{Mored-basic}  of
$\moredname$
for $\tau>0$ can be easily checked. Indeed,
for  the $\Gamma$-$\liminf$ estimate, observe that it is trivial in the case   $\conjz {z}{q}{\zeta} =\moreau uq{\eta} =0$, as
$\mathcal{M}_\eps^{\mathrm{red}}$ takes positive values for all $\eps>0$. Suppose now that  $\conjz {z}{q}{\zeta} +\moreau uq{\eta} >0$,
e.g.\ that  $\moreau uq{\eta} >0$.
Now, $(q_\eps,\eta_\eps) \to (q,\eta)$ implies that $\moreau u{q_\eps}{\eta_\eps} \geq \bar c >0$ for sufficiently small  $\eps $,
and from  \eqref{crucial-observation} we deduce that
\[
 \liminf_{\eps \downarrow 0} \Mered {q_\eps}{\tau_\eps}{(u_\eps',z_\eps')}{(\eta_\eps,\zeta_\eps)}=\infty=  \Mored q\tau{(u',z')}{(\eta,\zeta)} \,.
\]
The $\Gamma$-$\limsup$ estimate follows by taking the recovery sequence $(q_\eps\tau_\eps,u_\eps',z_\eps',\eta_\eps,\zeta_\eps) = (q,\tau,u',z',\eta,\zeta)$. In fact,
$\conjz {z}{q}{\zeta} +\moreau uq{\eta} >0$, then the
$\limsup$-inequality in \eqref{limsup} is trivial. If  $\conjz {z}{q}{\zeta} =\moreau uq{\eta} =0$, \eqref{limsup} can be checked straightforwardly.
%using again \eqref{crucial-observation}.
%%%
%\eqref{crucial-observation},
%%
%%%is trivial if. Suppose that this does not hold, and that, e.g.,
%%%$$.

For $\alpha=1$, in the case $\tau=0$,
  \eqref{crucial-observation} clearly yields the $\Gamma$-$\liminf$ estimate, whereas the $\Gamma$-$\limsup$ one can be obtained by with the recovery sequence
$(q_\eps,\tau_\eps,u_\eps',z_\eps',\eta_\eps,\zeta_\eps)=(q,\tau_\eps^*,u',z',\eta,\zeta) $   with
\[
\tau_\eps^* = \eps  \frac{\sqrt{\vpot zq{z'} +
 \vpot uq{u'} } }{\sqrt{
 \conjz {z}{q}{\zeta} +  \moreau uq\eta}}.
\]

For $\alpha>1$, in the case $\tau=0$,  the $\Gamma$-$\liminf$ estimate follows  taking into account that \eqref{crucial-observation} yields
\begin{equation}
\label{crucial-observation-bis}
\Mered q\tau{(u',z')}{(\eta,\zeta)}
\geq \frac{2}{\sqrt{\eps^{\alpha-1}}} \sqrt{\vpot zq{z'}  \moreau uq\eta}.
\end{equation}
Hence, if  both $\vpot zq{z'} >0 $  and  $\moreau uq\eta>0$, then $\liminf_{\eps \down 0} \Mered q\tau{(u',z')}{(\eta,\zeta)}   =\infty$.
In the case when either  $\vpot zq{z'} =0 $  or  $\moreau uq\eta=0$, we deduce the  $\Gamma$-$\liminf$ estimate from \eqref{crucial-observation}.
For the  $\Gamma$-$\limsup$ estimate, we again take the recovery sequence $(t,q,\tau_\eps^{**},u',z',\eta,\zeta) $, where now
\[
\tau_\eps^{**} = \eps  \frac{\sqrt{\vpot zq{z'} +
\eps^{\alpha-1} \vpot uq{u'} } }{\sqrt{
 \conjz {z}{q}{\zeta} +  \frac1{\eps^{\alpha-1}}\moreau uq\eta}}.
\]

The discussion of the case $\alpha
\in (0,1)$ is completely analogous, also in view of \eqref{specular}.

Finally, in order to prove \eqref{ioffe-refined}, we apply the Ioffe Theorem \cite{Ioff77LSIF}. For this,  we
introduce a functional $\overline{\calM} : [0,\infty) \times \calQ\times [0,\infty) \times \calQ \times \R^{n+m} \to [0,\infty] $ subsuming
 the functionals $\calM_\eps$ and
$\calM_0$, viz.\
\[
\overline{\calM} (\eps; q,\tau,q',\xi) := \left\{
\begin{array}{ll}
\Me \eps q \tau {q'} \xi & \text{if } \eps>0,
\\
\Mo  q \tau {q'} \xi & \text{if } \eps=0.
\end{array}
\right.
\]
Arguing in the very same way as in the proof of \cite[Lemma 3.1]{MRS09}, it  can be inferred that the functional $\overline{\calM}$ is lower semicontinuous
on $[0,\infty) \times \calQ\times [0,\infty) \times \calQ \times \R^{n+m} $,
 and that $(\tau,q') \mapsto \overline{\calM} (\eps; q,\tau,q',\xi) $ is convex for all $(\eps,q,\xi) \in  [0,\infty) \times \calQ \times \R^{n+m} $. Hence, the Ioffe Theorem ensures that
\[
\liminf_{\eps \down 0}
\int_0^S \overline{\calM}(\eps;{q_\eps(s)},{\tau_\eps(s)},{q_\eps'(s)},{\xi_\eps(s)}) \dd s
\geq \int_0^S \overline{\calM}(0;{q(s)},{\tau(s)},{q'(s)},{\xi(s)}) \dd s,
\]
whence \eqref{ioffe-refined}.
\end{proof}
%%%
%\begin{remark}
%\paragraph{\bf The contact set.}
%\beric discussion of the contact set, and of the chain rule \eric
Observe that the functional $\calM_0$ \eqref{basic-one}
fulfills for all $(q,\tau) \in \calQ\times [0,\infty)$
\begin{equation}
\label{mo-big}
\Mo{q}{\tau}{q'}{\xi} \geq \langle q', \xi \rangle = \langle u',\eta \rangle + \langle z', \zeta \rangle \qquad \text{for all } q' =(u',z') \in \calQ \text{ and all }
\xi = (\eta,\zeta )\in \R^{n +m}.
\end{equation}
Indeed,  for $\tau>0$, the inequality is trivial if either $\moreau uq\eta >0$ or $\conjz zq{\zeta}>0$. When both of them equal $0$, then
$\eta =0$ and
$
\langle q', \xi \rangle = \langle \zeta,z'\rangle  \leq  \calR_0(q,z')=  \Mo{q}{\tau}{q'}{\xi}$.
For $\tau =0$, e.g.\ in the case $\alpha >1$ we have, if $z'=0$,
\[
\langle q', \xi \rangle = \langle \eta,u'\rangle  \leq \sqrt{\langle \vcof u{q}{u'},u'\rangle  } \sqrt{\langle \vcofinv u{q}{\eta},\eta'\rangle  }
= \Mored{q}{\tau}{q'}{\xi} +0= \Mo{q}{\tau}{q'}{\xi}
\]
while, if $\eta=0$,
\[
\begin{aligned}
\langle q', \xi \rangle = \langle \zeta,z'\rangle &  = \langle \zeta - \omega, z'\rangle + \langle \omega, z'\rangle
\\
 & \leq \sqrt{\langle \vcof z{q}{z'},z'\rangle  } \sqrt{\langle \vcofinv z{q}{(\zeta{-}\omega)},(\zeta{-}\omega)\rangle  } + \calR_0(z')
= \Mo{q}{\tau}{q'}{\xi}
\end{aligned}
\]
where we have  chosen  $w\in \stab q$ such that $  \conjz zq{\zeta} = \moreau{z}{q}{\zeta{ -}\omega}  = \frac12 \langle \vcofinv{z}{q}{(\zeta{-}\omega)},(\zeta {-}\omega)\rangle$, and from the fact that $\langle \omega, z'\rangle  \leq  \calR_0(z')
$.

For the ensuing discussions, the  set where   \eqref{mo-big}  holds as an equality shall play a crucial role.
We postpone its precise definition right before the statement of Proposition \ref{prop:charact}, cf.\ \eqref{contact-set} ahead.
%%%%%%
\paragraph{\bf The vanishing-viscosity result.}
%We are now in the position to give the main result of this note. It
Theorem \ref{th:main} below
states that, up to a subsequence the
parameterized solutions $(\sft_\eps,\sfq_\eps)_\eps$ of the (Cauchy problems for the) viscous system \eqref{gen-grad-syst}, converge to a
parameterized curve $(\sft,\sfq)$, complying with the analog of the energy balance \eqref{enid-eps-param},  with $\mathcal{M}_0$ in place of $\mathcal{M}_\eps$.

We postpone after the proof of Theorem \ref{th:main} a thorough analysis of the notion of solution to the rate-independent system
\eqref{rip-limit} thus obtained. Let us instead mention in advance that the line of the argument for
proving the limiting  parameterized energy balance \eqref{enid-param-lim} is by now quite standard, cf.\ the proofs of \cite[Thm.\,3.3]{MRS09},
\cite[Thm.\,5.5]{MRS10}. In fact, the \emph{upper energy estimate} (i.e.\ the inequality $\leq$ for \eqref{enid-param-lim}) shall follow from lower semicontinuity arguments, based on the application of the Ioffe Theorem \cite{Ioff77LSIF}. The  \emph{lower energy estimate} $\geq	$ will instead ensue from the chain rule \eqref{chain-rule}.
We also point out that, for the compactness argument it is actually not necessary to start from parameterized curves for which
  estimate \eqref{normalization} holds, uniformly w.r.t.\ time. In fact, the uniform integrability of the sequence $(\sft_\eps', \sfq_\eps')_\eps$  is sufficient, cf.\ \eqref{uniform-integrab} below.
 \begin{theorem}
\label{th:main}
%Let $\alpha>0$.
Assume \eqref{ass:dissip-pot-R}, \eqref{ass:dissip-pot-Vz}, \eqref{ass:dissip-pot-Vu}, and \eqref{ass:E}.
Let $(q_\eps)_\eps\subset H^1 (0,T;\calQ)$ be a sequence of solutions to the Cauchy problem for \eqref{gen-grad-syst}.
%\beric + ass. on data, \eri
% fulfilling estimate \eqref{L1-est}
 % and
% such that  there exists $q_0 \in \calQ$ with
 % \begin{equation}
  %\label{energy-convergence-initial}
 % \ene{0}{q_\eps(0)} \to \ene 0{q_0} \quad \text{as } \eps \down 0.
 % \end{equation}
 Choose nondecreasing surjective parameterizations $\sft_\eps : [0,S_\eps] \to [0,T]$  and set
 $
 \sfq_\eps(s)= (\mathsf{u}_\eps(s),\mathsf{z}_\eps(s) ): = q_\eps(\sft_\eps(s))
 $
 for $s\in [0,S_\eps]$.  Suppose that $S_\eps \to S$ as $\eps \down 0$ up to a subsequence, and that
  there exist $q_0 \in \calQ$ and $m \in L^1 (0,S)$ such that
  $\sfq_\eps (0) \to q_0$, and
  \begin{equation}
  \label{uniform-integrab}
  m_\eps:= \sft_\eps'+|\sfq_\eps'| \weakto m \qquad
 \text{in } L^1(0,S) \text{ as } \eps \down 0.
  \end{equation}

Then,
there exist a  (not-relabeled) subsequence
and a parameterized curve $(\sft,\sfq) \in \AC ([0,S]; [0,T]\times \calQ)$
 such that  as $\eps\down 0$
 \begin{equation}
 \label{ascoli}
 (\sft_\eps,\sfq_\eps) \to (\sft,\sfq) \text{ in } \rmC^0 ([0,S];[0,T]\times \calQ),
 \end{equation}
 $\sft'+|\sfq'| \leq m$ a.e.\ in $(0,S)$,
 %\footnote{\beric it seems to me that, in this case we are no longer able to obtain
 %$\sft'(s) + |\sfq'(s)| = m(s)$ for a.a.\ $s\in (0,S)$, cf.\ the discussion at the end of the proof..
 %\eric}
 and $(\sft,\sfq)$ fulfills the (parameterized) energy identity \begin{equation}
\label{enid-param-lim}
\begin{aligned}
&
\ene {\sft(s_2)}{\sfq(s_2)} +\int_{s_1}^{s_2}
\Mo{\sfq(r)}{\sft'(r)}{\sfq'(r)}{-\rmD_q\ene{\sft(r)}{\sfq(r)}} \
\dd r
\\
& = \ene {\sft(s_1)}{\sfq(s_1)}
+ \int_{s_1}^{s_2} \partial_t \ene {\sft(r)}{\sfq(r)} \sft'(r) \dd  r \qquad \text{for all } 0 \leq s_1\leq s_2 \leq S.
\end{aligned}
\end{equation}
\end{theorem}
\begin{proof}
Up to a reparameterization, we  may suppose that the
curves $(\sft_\eps,\sfq_\eps)$ are defined on the fixed time interval $[0,S]$. We split the proof is three steps.

\noindent
\underline{Step $1$: compactness.} Observe that
for every $0 \leq s_1 \leq s_2 \leq S$
%the sequence $(\sfq_\eps)_\eps \subset \rmC^0 ([0,S];\calQ)$ is bounded. Indeed,
\begin{equation}
\label{analogue}
|\sfq_\eps(s_1) - \sfq_\eps (s_2)| \leq \int_{s_1}^{s_2} |\sfq_\eps'(s)| \dd s \leq \int_{s_1}^{s_2}m_\eps (s) \dd s \,.
\end{equation}
%for every $s\in [0,S]$ and, in turn,
Since $(\sfq_\eps (0))_\eps$ is bounded, we deduce from \eqref{analogue} that  $(\sfq_\eps)_\eps \subset \rmC^0 ([0,S];\calQ)$ is bounded as well.
What is more, as  the family $(m_\eps)_\eps$
 is uniformly integrable \eqref{uniform-integrab}, $(\sfq_\eps)_\eps$ complies with the equicontinuity condition of the Ascoli-Arzel\`a Theorem and so does
 $(\sft_\eps)_\eps$, by the analog of estimate \eqref{analogue}. Hence, \eqref{ascoli} follows.
 Taking into account that $\calE \in \rmC^1 ([0,T]\times \calQ)$, we immediately conclude  from \eqref{ascoli} that
 \begin{equation}
 \label{convergence-of-energies}
 \ene{\sft_\eps}{\sfq_\eps} \to \ene{\sft}{\sfq}, \qquad
 \rmD_q \ene{\sft_\eps}{\sfq_\eps} \to \rmD_q\ene{\sft}{\sfq},
 \qquad
 \partial_t  \ene{\sft_\eps}{\sfq_\eps} \to \partial_t\ene{\sft}{\sfq} \quad \text{uniformly on } [0,S].
 \end{equation}
%In view of the power control from \ref{ass:E},
%Since $\sup_{s\in [0,S]}|\partial_t  \ene{\sft_\eps(s)}{\sfq_\eps(s)}| \leq C$ uniformly w.r.t.\ $\eps$, we also gather
%\begin{equation}
%\label{convergence-of-powers}
 %\partial_t  \ene{\sft_\eps}{\sfq_\eps} \to \partial_t\ene{\sft}{\sfq} \qquad \text{in } L^1 (0,S).
%\end{equation}
 Furthermore,  \eqref{uniform-integrab} also yields that the sequences $(\sft_\eps')_\eps$ and $(\sfq_\eps')_\eps$ are uniformly integrable. Thus,
 by the Pettis Theorem,
 up to a further extraction we find
 \begin{equation}
 \label{l1-conv}
 \sft_\eps' \weakto \sft' \qquad \text{in } L^1 (0,S), \qquad \sfq_\eps' \weakto \sfq' \qquad  \text{in } L^1 (0,S;\calQ),
 \end{equation}
 whence $\sft'+|\sfq'| \leq m$ a.e.\ in $(0,S)$.

 \noindent
 \underline{Step $2$: upper energy estimate. }
 We now take the limit as $\eps \down 0$ of the (parameterized) energy-dissipation balance \eqref{enid-eps-param} for every $0 \leq s_1 \leq s_2 \leq S$:
 \begin{equation}
 \label{UEE}
 \begin{aligned}
 &
 \ene {\sft(s_2)}{\sfq(s_2)} +\int_{s_1}^{s_2}
\Mo{\sfq(r)}{\sft'(r)}{\sfq'(r)}{-\rmD_q\ene{\sft(r)}{\sfq(r)}}
\dd r
\\
&
\stackrel{(1)}{\leq} \lim_{\eps \down 0}\ene {\sft_\eps(s_2)}{\sfq_\eps(s_2)} +\liminf_{\eps \down 0}\int_{s_1}^{s_2}
\Me{\eps}{\sfq_\eps(r)}{\sft_\eps'(r)}{\sfq'_\eps(r)}{-\rmD_q\ene{\sft_\eps(r)}{\sfq_\eps(r)}} \
\dd r
\\
& = \lim_{\eps \down 0}\ene {\sft_\eps(s_1)}{\sfq_\eps(s_1)}
+  \lim_{\eps \down 0} \int_{s_1}^{s_2} \partial_t \ene{\sft_\eps(r)}{\sfq_\eps(r)} \sft_\eps'(r) \dd  r
\\
&
\stackrel{(2)}{ =} \ene {\sft(s_1)}{\sfq(s_1)}
+  \int_{s_1}^{s_2} \partial_t \ene{\sft(r)}{\sfq(r)} \sft'(r) \dd  r\,,
\end{aligned}
\end{equation}
where $(1) $ follows from the energy convergence in
\eqref{convergence-of-energies} and  the previously proved
 \eqref{ioffe-refined}, and  (2) from \eqref{convergence-of-energies}, again,  combined with the first  of \eqref{l1-conv}.
 This concludes the upper energy estimate.

\noindent
\underline{Step $3$: lower energy estimate.} We have for all $0 \leq s_1\leq s_2 \leq S$ that
%It follows from \eqref{mo-big} that
  \begin{equation}
 \label{LEE}
 \begin{aligned}
  & \ene {\sft(s_1)}{\sfq(s_1)}
+  \int_{s_1}^{s_2} \partial_t \ene{\sft(r)}{\sfq(r)} \sft'(r) \dd  r
\\
& \stackrel{(1)}{ = }  \ene {\sft(s_2)}{\sfq(s_2)} +   \int_{s_1}^{s_2}  \langle -\rmD_q \ene{\sft(r)}{\sfq(r)} , \sfq'(r) \rangle \dd r
\\
& \stackrel{(2)}{ \leq  }  \ene {\sft(s_2)}{\sfq(s_2)} +\int_{s_1}^{s_2}
\Mo{\sfq(r)}{\sft'(r)}{\sfq'(r)}{-\rmD_q\ene{\sft(r)}{\sfq(r)}}
\dd r\,,
 \end{aligned}
 \end{equation}
 where (1) follows from the chain rule, and (2) is due to inequality
 \eqref{mo-big}.  In this way, we conclude \eqref{enid-param-lim}.

Finally, combining \eqref{UEE} and \eqref{LEE} it is easy to deduce
that
 \[
 \lim_{\eps \down 0} \int_{s_1}^{s_2}
\Me{\eps}{\sfq_\eps(r)}{\sft_\eps'(r)}{\sfq_\eps'(r)}{- \rmD_q
  \ene{\sft_\eps(r)}{\sfq_\eps(r)}}   \dd r = \int_{s_1}^{s_2}
\Mo{\sfq(r)}{\sft'(r)}{\sfq'(r)}{-\rmD_q\ene{\sft(r)}{\sfq(r)}} \dd r
\]
for all $0\leq s_1 \leq s_2 \leq S$, whence $ \int_{s_1}^{s_2} \calR_0
(\sfq_\eps(r), \mathsf{z}_\eps'(r)) \dd r \to \int_{s_1}^{s_2} \calR_0
(\sfq(r), \mathsf{z}'(r)) \dd r $.
 \end{proof}

\paragraph{\bf Balanced Viscosity parameterized solutions.}
Let us now gain further insight into the notion of solution to system
\eqref{intro:rip-limit} arising from the vanishing-viscosity limit.
First of all, we fix its definition.

\begin{definition}\label{def:bv-param}
  Let $(\calR_0,\vpotname z,\vpotname u, \cE)$ comply with
  \eqref{ass:dissip-pot-R}, \eqref{ass:dissip-pot-Vz},
  \eqref{ass:dissip-pot-Vu}, and \eqref{ass:E}.  A curve $(\sft,\sfq)
  \in \mathrm{AC} ([0,S];[0,T]\times \calQ)$ is called a
  \emph{parameterized Balanced Viscosity} ($\pbv$, for short) solution
  to the rate-independent system $(\calQ, \calE, \calR_0 + \eps
  \vpotname z + \eps^\alpha \vpotname u) $ if $\sft:[0,S] \to [0,T]$
  is nondecreasing, and the pair $(\sft,\sfq)$ complies with the
  energy-dissipation balance \eqref{enid-param-lim} for all $0\leq
  s_1\leq s_2 \leq S$.

  Furthermore, $(\sft,\sfq)$ is called
  \begin{compactitem}
  \item  \emph{non-degenerate}, if
  \begin{equation}
  \sft'(s) + |\sfq'(s)| >0 \qquad \foraa\, s \in (0,S);
  \end{equation}
  \item \emph{surjective},  if $\sft: [0,S] \to [0,T]$ is surjective.
  \end{compactitem}
\end{definition}

\begin{remark}
\label{rmk:nondeg}
\upshape Observe that, even in the case when the function $m$ in
\eqref{uniform-integrab} is a.e.\ strictly positive, Theorem
\ref{th:main} does not guarantee the existence of non-degenerate
$\pbv$ solutions. However, any degenerate $\pbv $ solution
$(\sft,\sfq)$ can be reparameterized to a non-degenerate one
$(\tilde{\sft}, \tilde{\sfq}) : [0,\tilde{S}] \to [0,T]\times \calQ$,
even fulfilling the \emph{normalization condition}
 \begin{equation}
 \label{fake-norma}
 \tilde{\sft}'(\sigma) + \tilde{\sfq}'(\sigma)=1 \qquad \foraa\,
 \sigma \in (0,\tilde S)\,.
\end{equation}
Indeed, following \cite[Rmk.\ 2]{MRS09}, starting from a (possibly
degenerate) solution $(\sft,\sfq)$, we set
 \[
 \sigma(s):= \int_0^s \sft'(r) + |\sfq'(r)| \dd r \qquad \text{and }
 \tilde{S}:= \sigma (S),
 \]
 and define $(\tilde{\sft}(\sigma), \tilde{\sfq}(\sigma)):=
 (\sft(s),\sfq(s))$ if $\sigma = \sigma(s)$. Then, the very same
 calculations as in \cite[Rmk.\ 2]{MRS09} lead to \eqref{fake-norma}.
 \end{remark}

 We conclude this section with a characterization of $\pbv$ solutions
 in the same spirit as \cite[Prop.\ 2]{MRS09} and \cite[Prop.\
 5.3]{MRS10}, \cite[Cor.\ 4.5]{mielke-rossi-savare2013}.  We show that
 the energy identity \eqref{enid-param-lim} defining the concept of
 $\pbv$ solutions is equivalent to the corresponding energy inequality
 on the interval $[0,S]$, and to the energy inequality in a
 differential form.  Finally, \eqref{char-contset} below provides a
 further reformulation of this solution concept which involves the
 \emph{contact set} (cf.\ \cite{MRS10, mielke-rossi-savare2013})
\begin{equation}
\label{contact-set}
\contset{q}:= \{ (\tau,q',\xi)\in [0,\infty) \times \calQ \times
\R^{n+m} \, : \Mo{q}{\tau}{q'}{\xi} = \langle q',\xi\rangle \}
\end{equation}
Observe that for all $q \in \calQ$ the set $\contset{q}$ is closed, as
the functional $\Mo{q}{\cdot}{\cdot}{\cdot}$ is lower semicontinuous.
In Proposition we will provide \ref{l:cont-set} the explicit
representation of $\contset{q}$. This and \eqref{char-contset} we will
be at the core of the reformulation of $\pbv$ solutions in terms of
subdifferential inclusions, which we will discuss in Sec.\ \ref{s:4}.

\begin{proposition}
\label{prop:charact}
Let $(\calR_0,\vpotname z,\vpotname u, \cE)$ comply with
\eqref{ass:dissip-pot-R}, \eqref{ass:dissip-pot-Vz},
\eqref{ass:dissip-pot-Vu}, and \eqref{ass:E}.  A curve $(\sft,\sfq)
\in \mathrm{AC} ([0,S];[0,T]\times \calQ)$, with $\sft$ nondecreasing,
is a $\pbv$ solution to the rate-independent system $(\calQ,\calE,
\calR_0 + \eps \vpotname z+ \eps^\alpha \vpotname u)$ if and only if
one of the following equivalent conditions is satisfied:
 \begin{enumerate}
\item  \eqref{enid-param-lim} holds as an inequality on $(0,S)$, i.e.
\[
\begin{aligned}
&
\ene {\sft(S)}{\sfq(S)} +\int_{0}^{S}
\Mo{\sfq(r)}{\sft'(r)}{\sfq'(r)}{-\rmD_q\ene{\sft(r)}{\sfq(r)}} \
\dd r
\\
& \leq  \ene {\sft(0)}{\sfq(0)}
+ \int_{0}^{S} \partial_t \ene {\sft(r)}{\sfq(r)} \sft'(r) \dd  r;
\end{aligned}
\]
\item the above energy inequality holds in the differential form
  $\frac{\dd}{\dd s} \ene{\sft}{\sfq} +
  \Mo{\sfq}{\sft'}{\sfq'}{-\rmD_q\ene{\sft}{\sfq}} \leq \partial_t
  \ene {\sft}{\sfq} \sft' $ a.e.\ in $(0,S)$;
\item the triple $(\sft',\sfq',-\rmD_q \ene{\sft}{\sfq})$ belongs to
  the contact set, i.e.
\begin{equation}
\label{char-contset}
(\sft'(s),\sfq'(s),-\rmD_q \ene{\sft(s)}{\sfq(s)}) \in \contset{\sfq(s)} \qquad \foraa\, s \in (0,S).
\end{equation}
\end{enumerate}
\end{proposition}
The proof of Proposition \ref{prop:charact} is omitted: it follows by
exploiting the chain rule \eqref{chain-rule}, with arguments akin to
those in the proof of Theorem \ref{th:main}, see also
\cite[Prop.\,2]{MRS09} and \cite[Prop.\,5.3]{MRS10},
\cite[Cor.\,4.5]{mielke-rossi-savare2013}.

\section{Physical interpretation}
\label{s:4}

The following result provides a thorough description of the (closed)
contact set $\contset{q}$, cf.\ \eqref{contact-set}.  As we will see,
the representation of $\contset q$ substantially different in the
three cases $\alpha>1$, $\alpha=1$, and $\alpha\in (0,1)$. That is
why, in Proposition \ref{l:cont-set} below we will use the notation
$\Sigma_{\alpha>1}(q)$, $\Sigma_{\alpha=1}(q)$, and $\Sigma_{\alpha
  \in (0,1)}(q)$.  We will prove that these sets are given by the
union of subsets describing the various evolution regimes for the
variables $u$ and $z$. The notation for these subsets will be of the
form
\[
\rgm{A}{r}{B}{s} \qquad \text{with } \mathrm{A},  \mathrm{B} \in \{ \mathrm{E}, \mathrm{R},  \mathrm{V}, \mathrm{B} \} \text{ and } \mathsf{r}, \mathsf{s} \in \{   \mathsf{u},
 \mathsf{z} \}.
\]
The letters $\mathrm{E}, \mathrm{R},  \mathrm{V}, \mathrm{B}$ stand for \emph{Equilibrated}, \emph{Rate-independent},
\emph{Viscous}, and \emph{Blocked}, respectively. For instance,
$\rgm EuRz$ is the set of $(\tau,q',\xi)$ corresponding to equilibrium for $u$ and rate-independent evolution for $z$, cf.\ \eqref{ssu} below; we postpone more comments after the statement of
Proposition \ref{l:cont-set}.
Observe that all of these sets depend on the state variable $q$, as does $\contset q$. However, for simplicity
we will not highlight this in their notation.
In their description %of these sets,
%Here, as well as in the previous sections,
we  shall always refer to the representation
$q'=(u',z')$ for the velocity variable, and
$\xi =(\eta,\zeta)$ for the force variable.
\begin{proposition}
\label{l:cont-set}
Assume \eqref{ass:dissip-pot-R}, \eqref{ass:dissip-pot-Vz}, \eqref{ass:dissip-pot-Vu}, and \eqref{ass:E}.
Then, for
\begin{description}
\item[\underline{$\alpha >1$}] the contact set is given by
 \begin{equation}
  \label{cont-set-ap}
  \Sigma_{\alpha>1}(q)  =  \rgm EuRz \cup \rgm VuBz \cup \rgm EuVz
 \end{equation}
  where
  \begin{align}
  &   \label{ssu}
  \rgm EuRz: = \{ (\tau,q',\xi)\, :
  \ \tau>0,\ \ q'=(u',z'), \ \xi = (0,\zeta) \text{ and
  }   \partial\calR_0 (q,z') \ni \zeta \},   \\
&\label{fastu}
 \rgm VuBz : = \{ (\tau,q',\xi)\, : \ (\tau,q',\xi)
  =(0,(u',0),(\eta,\zeta)) \text{ and }  \exists\, \thn u \in [0,1]\, :
   \ \thn u \vcof u{q}u'= (1-\thn u) \eta \}, \\
&\label{equilu}
 \begin{aligned}
 \rgm EuVz:= \{ (\tau,q',\xi)\, :  \  & \tau =0, \ q'=(u',z'), \  \xi
 = (0,\zeta) \text{ and } \\
 & \exists\,   \thn z \in [0,1]  \, : \
(1-\thn z)  \partial\calR_0 (q,z') + \thn  z
    \vcof z{q}z' \ni  (1-\thn z) \zeta \}.
  \end{aligned}
 \end{align}
\item[\underline{$\alpha =1$}] the contact set is given by
  \begin{equation}
   \label{cont-set-au}
    \Sigma_{\alpha=1}(q)  =  \rgm EuRz \cup \rgm VuVz %\rgm EyVz
  \end{equation}
  where
  \begin{align}
   &\label{jumpset}
    \rgm VuVz: =\left \{ (\tau,q',\xi)\, :
    \ \tau=0,   \text{ and } \exists\, \theta \in[0,1]\, : \
    \Big\{ \begin{array}{l}
         \theta \vcof u{q}u' = (1-\theta) \eta,  \\
        (1-\theta)  \partial\calR_0 (q,z') +   \theta
        \vcof z{q}z' \ni (1-\theta) \zeta
      \end{array} \right\}.
 \end{align}
\item[\underline{$\alpha \in (0,1)$}] the contact set is given by
  \begin{equation}
    \label{cont-set-am}
    \Sigma_{\alpha \in (0,1)}(q)  =  \rgm EuRz \cup \rgm BuVz \cup \rgm VuRz
  \end{equation}
  with % $\stsl q$ from    \eqref{ssu} and
  \begin{align}
    &\label{fastz}
     \begin{aligned}
       \rgm BuVz:= \{ (\tau,q',\xi): \ & \tau =0, \ q'=(0,z'), \  \xi =
       (\eta,\zeta) \text{ and } \\
      & \exists\, \thn z  \in [0,1]: \;
      (1{-}\thn z)\partial\calR_0 (q,z') + \thn  z
       \vcof z{q}z' \ni  (1{-}\thn z)\zeta \},
    \end{aligned}
  \\
  & \label{equilz}
   \rgm VuRz
    :=\left \{ (\tau,q',\xi): \; 
      (\tau,q',\xi)=(0,(u',z'),(\eta,\zeta)) \text{ and }
      \Big\{  \begin{array}{l}
    \exists\, \thn u \in [0,1]: \; \thn u \vcof u{q}u'= (1{-}\thn u) \eta,
    \\
  \partial\calR_0 (q,z') \ni \zeta
    \end{array}
 \right \}.
 \end{align}
\end{description}
\end{proposition}
\noindent
As \eqref{char-contset} reveals, the contact set encompasses all the
relevant information on the evolution of a \emph{parameterized
  Balanced Viscosity} solution.  The form of the sets $\rgm EuRz, \,
\rgm VuBz\, \ldots$ which constitute it is strictly related to the
mechanical interpretation of $\pbv$ solutions which shall be explored
at the end of this section. Let us just explain here that
\begin{compactitem}
\item the set $\rgm EuRz$ corresponds to equilibrium for the variable
  $u$ (as $\eta=0$), and a \emph{stick-slip} regime for $z$, which
  evolves rate-independently as expressed by $ \partial\calR_0 (q,z')
  \ni \zeta$.  Observe that the stationary state $u'=z'=0$ is also
  encompassed.
\item  The set $\rgm Vu Bz$ corresponds to the case in which the variable $u$
% ($z$, respectively)
  still has to relax to an equilibrium and thus is governed by a
  \emph{fast} dynamics at a jump $\tau=0$, while $z$ is ``blocked by
  viscosity'' and thus stays constant ($z'=0$).
%%
%%is the \emph{Stick-Slip} set, corresponding to the situation in which either the system is stationary (i.e.\
%%), or it is in a \emph{sliding regime}, i.e.\ the evolution is purely rate-independent, as .
\item The set $\rgm EuVz$ corresponds to the regime in which $z$
  evolves according to viscosity at a jump $\tau=0$, and $u$ follows
  $z$ in such a way that it is at an equilibrium ($\eta=0$).
%  %\item
% The set $\equil uq$ corresponds to the situation in which the variable $u$ is already in \emph{equilibrium}, as its corresponding force $\eta =0$.
\item The set $\rgm VuVz$ corresponds to the case where the evolution
  of the system at a jump $\tau=0$ is governed by viscosity both in
  $u$ and in $z$.
\item The set $ \rgm BuVz $ encompasses the case in which the variable
  $z$ at a jump $\tau=0$ evolves according to viscosity, while $u$ is
  blocked by viscosity ($u'=0$).
\item The set $ \rgm VuRz $ describes viscous evolution for $u$ and
  rate-independent evolution for $z$.
\end{compactitem}

\begin{remark}\upshape
\label{rmk:specular}
Let us stress once more that, as mentioned in advance, in the
vanishing-viscosity limit the evolution regimes for $\alpha>1$ and
$\alpha \in (0,1)$ mirror each other. Indeed, formulae
\eqref{cont-set-ap} and \eqref{cont-set-am} are specular, up to
observing that the analog of the equilibrium regime
$\mathrm{E}_{\mathsf{u}}$ is indeed the rate-independent regime
$\mathrm{R}_{\mathsf{z}}$,  see also Figure \ref{fig:SwitchRegimes}. 
\end{remark}

\begin{proof}[Proof of Proposition \ref{l:cont-set}]
In all the three cases $\alpha>1$, $\alpha =1$, and $\alpha \in (0,1)$,
for $\tau>0$ the contact condition
$\Mo{q}{\tau}{q'}{\xi} = \langle \xi, q'\rangle$
%(with
%$q'=(u',z')$ and
%$\xi=(\eta,\zeta)$)
 can hold only if  the constraints
$\eta =0$ and $\zeta \in \stab  q$ are satisfied. Then, $\Mo{q}{\tau}{q'}{\xi} = \langle \xi, q'\rangle$ reduces to
$\calR_0 (q,z') = \langle \zeta,z'\rangle $. Since $\zeta \in \stab  q$, this is equivalent to
 $\zeta \in  \partial\calR_0 (q,z')$ by \eqref{charact-1-homog}.
This gives the  set
$\rgm EuRz$, which contributes to the contact set $\contset q$
in the three cases $\alpha>1$, $\alpha=1$, and $\alpha \in (0,1)$.

For $\alpha=1$, observe that in the case $\tau=0$ the contact condition is
\begin{equation}
\label{case-alpha1}
\calR_0(z') + 2\sqrt{\vpot zq{z'} + \vpot uq{u'} }\sqrt{\conjz
  zq{\zeta} + \moreau uq{\eta} } = \langle \zeta, z'\rangle + \langle
\eta, u'\rangle.
\end{equation}
Let us first address the case in which $\sigma_1:= \sqrt{\vpot
  zq{z'}+\vpot uq{u'} }=0$ or $\sigma_2:= \sqrt{\conjz
  zq{\zeta}+\moreau uq{\eta} } =0$. The former case corresponds to the
stationary state $u'=z'=0$, which means $\theta=1$ in \eqref{jumpset}.
The latter to $\conjz zq\zeta =0$ (if and only if $\zeta \in \stab q$)
and $\eta=0$ Hence \eqref{case-alpha1} becomes $\calR_0(z')= \langle
\zeta, z'\rangle$, whence $\zeta \in \partial\calR_0 (q,z')$ by
\eqref{charact-1-homog}, again.  This corresponds to $\theta=0$ in
\eqref{jumpset}.  If $\sigma_1 \sigma_2>0$, then we rewrite $2\sigma_1
\sigma_2$ as $\lambda \sigma_1^2 + \frac1\lambda \sigma_2^2$, with
$\lambda>0$ given by $\lambda = \frac{\sigma_2}{\sigma_1}$. With such
$\lambda$ \eqref{case-alpha1} rewrites as
\[
\calR_0(z') + \lambda ( \vpot zq{z'} + \vpot uq{u'}) + \frac1{\lambda}
( \conjz zq{\zeta}+\moreau uq{\eta}) = \langle \zeta, z'\rangle +
\langle \eta, u'\rangle.
\]
Upon multiplying both sides by $\lambda$, using that $\vpotname z$ and
$\vpotname u$ are positively homogeneous of degree $2$, and
rearranging terms, we get
\[
\calR_0(z') +  \vpot zq{ \lambda z'} +  \conjz zq{\zeta} - \langle \zeta,\lambda  z'\rangle
 =  \langle \eta, \lambda u'\rangle -  \vpot uq{\lambda u'}  - \moreau uq{\eta}.
\]
By the Fenchel-Moreau equivalence, this gives
\[
  \begin{array}{ll}
  &
    \vcof u{q}{(\lambda u')} = \eta,  \\
&    \partial\calR_0 (q, \lambda z') +
    \vcof z{q}{(\lambda z')} \ni  \zeta
    \end{array}
\]
with $\lambda >0$. Then, \eqref{jumpset}
follows with $\theta \in (0,1)$ such that $\lambda = \frac{\theta}{1-\theta}$.
 All in all, for
 $\alpha =1$ we have proved that, if $(\tau,q',\xi) \in \Sigma_{\alpha=1}(q)$, then either  $(\tau,q',\xi) \in \rgm EuRz$, or
 $(\tau,q',\xi) \in \rgm VuVz$. This concludes  the proof of \eqref{cont-set-au} for $\Sigma_{\alpha=1}(q)$.

In the case $\alpha>1$ and $\tau=0$, $\Mo{q}{\tau}{q'}{\xi} $ is finite if and only if either $z'=0$, or $\eta=0$. In the former case, the contact condition reduces to
$\sqrt{\langle \vcof uq{u'}, u' \rangle} \sqrt{\langle \vcofinv uq{\eta}, \eta \rangle} = \langle \eta, u'\rangle$, which is equivalent to the fact that there exists $\thn u \in [0,1]$ with $\thn u \vcof u{q}u'= (1-\thn u) \eta$.  This yields the set $\rgm VuBz$. In the latter case,
the contact condition rephrases as
\[
\calR_0(q,z') + \sqrt{\langle \vcof zq{z'}, z' \rangle} \sqrt{\langle \vcofinv zq{(\zeta{-}\omega)}, \zeta{-}\omega \rangle}  = \langle \zeta , z'\rangle =   \langle \omega,z'\rangle+
\langle \zeta{-}\omega,z'\rangle,
\]
 with $\omega \in \stab q$ such that $\conjz zq{\zeta} = \frac12 \langle \vcofinv zq{(\zeta{-}\omega)}, \zeta{-}\omega \rangle$. It is immediate to check that the above chain of equalities implies
 \[
 \left\{
 \begin{array}{ll}
 \omega \in \partial \calR_0(q,z'),
 \\
 (1-\thn z )(\zeta{-}\omega )= \thn z     \vcof zq{z'} \quad \text{for some } \thn  z \in [0,1].
 \end{array}
 \right.
 \]
This yields the set $\rgm EuVz$.
All in all,
in the case $\alpha>1$
we have proved that, if $(\tau,q',\xi) \in \Sigma_{\alpha>1}(q)$, then either  $(\tau,q',\xi) \in \rgm EuRz$, or
 $(\tau,q',\xi) \in \rgm VuBz$, or  $(\tau,q',\xi) \in \rgm EuVz$. This concludes  \eqref{cont-set-ap}.

The proof of  \eqref{cont-set-am} follows the very same lines and is thus omitted.
\end{proof}

The \underline{\bf main result} of this paper is the following
theorem, which is in fact a direct consequence of the characterization
\eqref{char-contset} of $\pbv$ solutions in terms of the contact set,
and of Proposition \ref{l:cont-set}.  Observe that, we confine
ourselves to \emph{non-degenerate} $\pbv$ solutions only. This is not
restrictive, in view of Remark \ref{rmk:nondeg}.

\begin{theorem}[Reformulation as a system of  subdifferential inclusions]
\label{prop:diff-incl}
Assume \eqref{ass:dissip-pot-R}, \eqref{ass:dissip-pot-Vz},
\eqref{ass:dissip-pot-Vu}, and \eqref{ass:E}.  A curve $(\sft,\sfq)
\in \AC([0,S]; [0,T]\times \calQ)$ with nondecreasing $\sft$ is a
\emph{non-degenerate} parameterized Balanced Viscosity solution to the
rate-independent system $(\calQ,\calE, \calR_0 + \eps \vpotname z+
\eps^\alpha \vpotname u)$ if and only if $\sft' + |\sfq'|>0$ a.e.\ in
$(0,S)$ and there exist two Borel functions $\thn u, \, \thn z: [0,S]
\to [0,1]$ such that
the pair $(\sft,\sfq)$ with $\sfq= (\sfu,\sfz)$ satisfies  the
system of equations for a.a.\ $s\in (0,S)$: 
\begin{equation}
\label{diff-syst}
\begin{aligned}
&
\thn u(s)\, \vcof u{\sfq(s)}{\sfu'(s)} + (1{-}\thn{u} (s))\, \rmD_u
\enet{\sft(s)}{\sfu(s)}{\sfz(s)} \ni 0,
\\
&
(1{-}\thn{z} (s))\,  \partial\calR_0 (\sfq(s),\sfz'(s)) +
\thn z(s)\, \vcof z{\sfq(s)}{\sfz'(s)} +  (1{-}\thn{z} (s))\,  \rmD_z
\enet{\sft(s)}{\sfu(s)}{\sfz(s)} \ni 0,
\end{aligned}
\end{equation}
with
\begin{equation}
\label{switching}
\sft'(s)\,  \thn u(s) =  \sft'(s) \, \thn z (s) =0
\end{equation}
and the  following additional conditions depending on $\alpha$: 
\begin{align}
&&&\text{\underline{$\alpha>1$:}}&
  \label{add-con-pu} \thn u (s) \,(1 {-}\thn z (s)) =0;
   &&\\
&&&\text{\underline{$\alpha=1$:}} &
  \label{add-con-uu} \thn u(s) = \thn z (s); &&\\
&&&\text{\underline{$\alpha\in (0,1)$:}} &
   \label{add-con-mu} \thn z (s)\, (1{-}\thn u (s)) =0 .&&
\end{align}
\end{theorem}

 Figure \ref{fig:SwitchRegimes} displays the structure of the
allowed values for the parameters $(t',\thn u,\thn z)$ depending on $\alpha$.

\begin{figure}[h]%%%%%%%%%%% NEW 20.6.2014
\centerline{\unitlength1cm
\begin{picture}(0,0)
\put(1.93,1){
    \put(-1.2,-1){$t'$}\put(2.3,.1){$\thn u$}\put(-0.4,1.8){$\thn z$}
    \put(-1.8,-0.4){$\rgm EuRz$} \put(0.6,-0.4){$\rgm VuRz$}
    \put(1.85,1.1){$\rgm BuVz$} }
\put(7.3,1){
     \put(-1.2,-1){$t'$}\put(2.3,.1){$\thn u$}\put(-0.4,1.8){$\thn z$}
     \put(-1.8,-0.4){$\rgm EuRz$} \put(0.3,1.2){$\rgm VuVz$} }
\put(12.6,1){
     \put(-1.2,-1){$t'$}\put(2.3,.1){$\thn u$}\put(-0.4,1.8){$\thn z$}
     \put(-1.8,-0.4){$\rgm EuRz$} \put(0.15,0.75){$\rgm EuVz$}
     \put(0.5,1.9){$\rgm VuBz$} }
\end{picture}%
\includegraphics[width=15\unitlength]{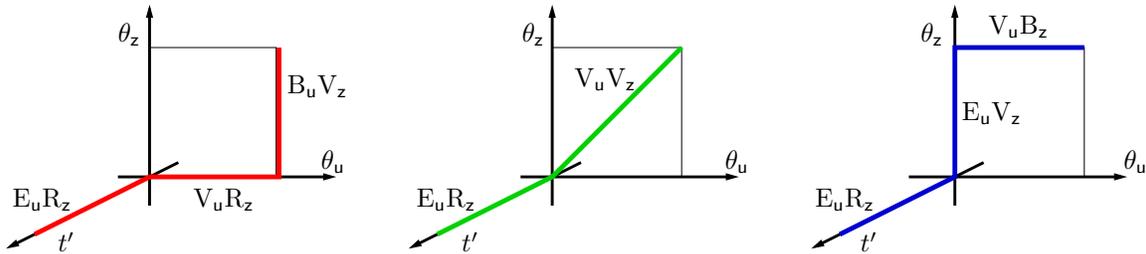}}
\caption{The switching between the different regimes, depending on the
  cases $\alpha<1$, $\alpha=1$,
  and $\alpha>1$, are displayed via the allowed combinations of the
  triples $(t',\thn u,\thn z)$.}
\label{fig:SwitchRegimes}
\end{figure}

\begin{remark}
\label{once-more-spec}
\upshape Observe that the conditions \eqref{add-con-pu} and
\eqref{add-con-mu} are specular (cf.\ Remark \ref{rmk:specular}),
revealing once more that the evolution regimes for $\alpha>1$ and
$\alpha<1$ reflect each other. Nonetheless, a major difference occurs
in that, under suitable conditions, for $\alpha>1$ the regime $\rgm
VuBz$ only occurs at the beginning, when $u$ relaxes fast to
equilibrium, cf.\ Proposition \ref{prop:alex}.
\end{remark}
%\begin{proof}
%Let us only comment on the case $\alpha=1$. Indeed, since $(\sft,\sfq)$ is \emph{non-degenerate}, in the case
%$\sft'=0$ either $u'\neq 0$, or $z'\neq 0$. Therefore, the coefficient $\theta$  in \eqref{jumpset} cannot be equal to $1$. Again recalling
%\eqref{jumpset}, we conclude \eqref{add-con-uu}.
%\end{proof}

Finally, let us get further insight into the mechanical interpretation of system \eqref{diff-syst}, with the constraints
\eqref{switching} and
 \eqref{add-con-pu}--\eqref{add-con-mu}.
Preliminarily, let us point out that, as in the case of parameterized solutions to the rate-independent system
\begin{equation}
\label{simple-rip}
\partial \calR_0(z(t),z'(t)) +\rmD_q \calI(t,z(t)) \ni 0 \qquad \text{in } (0,T),
\end{equation}
in the \emph{sole} variable $z$, $\sft'(s)=0$ if and only if the
system is jumping in the (slow) external time scale.  Therefore, from
\eqref{switching} we gather that, in all of the three cases
$\alpha>1$, $\alpha=1$, and $\alpha \in (0,1)$, when the system does
not jump, then it is either in the \emph{sticking} regime
(i.e. $\mathsf{u}'= \mathsf{z}'=0$), or in the \emph{sliding regime},
namely the evolution of $\mathsf{z}$ is purely rate-independent
(i.e. $\partial\calR_0 (\sfq,\mathsf{z}') + \rmD_z \ene
{\sft}{\mathsf{q}} \ni 0$), and $\mathsf{u}$ follows $\mathsf{z}$ in
such a way that it is at an equilibrium (i.e. $-\rmD_u \ene
{\sft}{\mathsf{q}}=0$).  It is the description of the system behavior
at jumps that significantly differs for $\alpha>1$, $\alpha=1$, and
$\alpha\in (0,1)$. \medskip

%%%%%
\paragraph{\bfseries \underline{Case $\alpha>1$:} fast relaxation of $u$.}
Here $\mathsf{u}$ relaxes faster to equilibrium than $\mathsf{z}$.
With \eqref{switching} and \eqref{add-con-pu} we are imposing at a
jump that either $\mathsf{z}'=0$ (which follows from $\thn z=1 $,
i.e.\ $\rgm VuBz$) or $\mathsf{u}$ is at equilibrium (corresponding to
$\thn u =0$, i.e.\ $\rgm EuVz$). In fact, $\mathsf{z}$ cannot change
until $\mathsf{u}$ has relaxed to equilibrium.
%$u$ relaxes,  $z$ is stuck ($z'=0$,
When $\mathsf{u}$ has reached the equilibrium, then $\mathsf{z}$ may
have either a \emph{sliding jump} (i.e. $\thn z =0$), or a
\emph{viscous jump} ($\thn z \in (0,1)$).

Our next result shows that, in fact, under the  condition  that the
energy $\calE$ is uniformly convex with respect to the variable $u$
(cf.\ Proposition \ref{prop:aprio-eps}), after an initial phase in
which $\sfz$ is constant and $\mathsf{u}$ relaxes to an equilibrium
evolving by viscosity (i.e.\ the  solution  is in regime
$\rgm VuBz$), $\sfu$ never leaves the equilibrium afterwards.  In
that case  the evolution of the system is completely described by
$\mathsf{z}$, which turns out to be a parameterized Balanced Viscosity
solution to the rate-independent system driven by the \emph{reduced
energy functional} obtained minimizing out the variable $u$.

\begin{proposition}
\label{prop:alex}
Assume \eqref{ass:dissip-pot-R}, \eqref{ass:dissip-pot-Vz},
\eqref{ass:dissip-pot-Vu}, and \eqref{ass:E}.  Additionally, suppose
that  $\calE$ complies  with \eqref{en-plus}, and denote by
$u=M(t,z)$ the unique solution of $\rmD_u\calE(t,u,z)=0$, i.e.\ the
minimizer of $\calE(t,\cdot,z)$.  Let $(\sft,\sfq) \in \AC ([0,S];
[0,T]\times \calQ )$ be a parameterized Balanced Viscosity solution to
the rate-independent system $(\calQ,\calE, \calR_0 + \eps \vpotname z+
\eps^\alpha \vpotname u)$ with $\alpha>1$.  Set
\begin{equation}
\label{eq:frakS}
\mathfrak S:= \{ s \in [0,S]\, : \ \rmD_u \ene {\sft (s)}{\sfq (s)} =0\}.
\end{equation}
Then, $\mathfrak S$ is either empty or it has the form $[s_*, S]$ for
some $s_*\in [0,S]$.

(a) Assume $s_*>0$, then for $s\in [0,s_*)=[0,S]\setminus \mathfrak S$
we have $\sft(s)=\sft(0)$
and $\sfz(s)= \sfz(0)$, whereas $\sfu$ is a solution to the
reparameterized the gradient flow for
$(\R^n,\calE(\sft(0),\cdot,\sfz(0)), \bbV_\sfu)$ (regime $\rgm VuBz$), namely
\begin{equation}
\label{eq:GS-u-only}
0 = \thn u(s) \, \bbV_\sfu (\sfu(s),\sfz(0)) 
 \dot\sfu(s) + (1{-}\thn u(s))\, \rmD_u
\calE(\sft(0),\sfu(s),\sfz(0)) \quad \text{with }\sfu(0) \neq
M(\sft(0),\sfz(0)).
\end{equation}

(b) Assume $\mathfrak S=[s_*,S]$ with $s_*<S$, then for $s\in
[s_*,S]$ we have $\sfu(s)=M(\sft(s),\sfz(s))$ whereas the pair
$(\mathsf{t},\mathsf{z})$ is a parameterized Balanced Viscosity
solution to the reduced rate-independent system $(\R^m, \calI, \calR_0 +
\eps \vpotname z)$ with the
\emph{reduced energy functional } $\calI : [0,T] \times \R^m \to
\R;  (t,z) \mapsto \min_{u \in \R^n} \en
t{u}{z} = \calE(t,M(t,z),z)$, which corresponds to the regimes $\rgm
EuVz$ and $\rgm EuRz$.
\end{proposition}
\begin{proof}
  To avoid overloaded notation we will often omit the state-dependence
  of the functions $\bbV_\sfu$ and $\bbV_\sfz$.  For easy reference we
  repeat all the conditions for a $\BV$ solution $(\sft,\sfq)$ (cf.\
  Theorem \ref{prop:diff-incl}), in the case $\alpha>1$: 
\begin{align*}
&\text{(i)} \quad 0=\thn u \bbV_\sfu \sfu' + (1{-}\thn u) \rmD_u
\calE(\sft,\sfu,\sfz),\qquad
\text{(ii)}\quad  0\in (1 {-}\thn z)\partial\calR_0(\sfq,\sfz')+ \thn z
\bbV_\sfz \sfz'  + (1{-}\thn z) \rmD_z \calE(\sft,\sfu,\sfz),\\
&\text{(iii)}\quad \sft'\thn u=0,\qquad \text{(iv)} \quad \sft'\thn z
=0, \qquad \text{(v)} \quad  \thn u\,(1{-}\thn z)=0, \qquad
\text{(vi)} \quad \sft'+|\sfu'|+|\sfz'|>0,
\end{align*}
which have to hold for a.a.\ $s\in  (0,S). $

 \underline{Step 1:} By the continuity of $(\sft ,\sfz)$  and
  $\rmD_u\calE$ the set $\mathfrak S$ is closed, hence its complement
  is relatively open. Consider an interval $(s_1,s_2)$ not
  intersecting with $\mathfrak S$.  Using (i) we find $\thn u>0$ a.e.\
  in $(s_1,s_2)$. 
  Hence, (iii) implies $\sft'=0$ a.e., and we obtain
  $\sft(s)=\sft(s_1)$ for $s\in [s_1,s_2]$. By (v) we find $\thn z=1$
  a.e.\ Now, (ii) implies $z'=0$ a.e., which implies
  $\sfz(s)=\sfz(s_1)$ for $s\in [s_1,s_2]$. From (vi) we conclude
  $\sfu'\neq 0$ a.e. Thus, we summarize
\[
\sft(s)=\sft(s_1), \quad \sfz(s)=\sfz(s_1), \quad 0 =
\bbV_u(\sfu(s),\sfz(s_1)) \sfu'(s) + \lambda(s)
\rmD_u\calE(\sft(s_1),\sfu(s), \sfz(s_1)),
\]
where $\lambda(s)=(1{-}\thn u(s))/\thn u(s) \in (0,\infty)$ a.e. In
particular, $\sfu$ satisfies \eqref{eq:GS-u-only}.  From
$u\in \AC([0,S];\R^m)$ and (i) we obtain $\lambda\in
L^1(s_1,s_2)$. Setting $\tau(s)=\int_{s_1}^s \lambda(\sigma)\dd \sigma$
and defining  the inverse $\hat s$ via $s=\hat{s}(\tau)$ 
we find $\hat s'(\tau)>0$ and 
$\hat s \in W^{1,1}(0,\tau(s_2))$. Moreover, the function $\hat u:
\tau\mapsto \sfu(\hat s(\tau))$ is a solution of the gradient flow
\begin{equation}
  \label{eq:GS-u-resc}
 0 = \bbV_u(\hat u(\tau),\sfz(s_1)) \hat u{}'(\tau) +
\rmD_u\calE(\sft(s_1),\hat u(\tau), \sfz(s_1)).
\end{equation} 
Furthermore,  we see that $s\mapsto \calE(\sft(s_1),\sfu(s),\sfz(s_1))$ is
strictly decreasing on $[s_1,s_2]$, since its time derivative is given
by $-\langle \sfu'(s),\bbV_u\sfu'(s)\rangle/\lambda(s)$ which is
negative a.e.

\underline{Step 2:} Since $\mathfrak S$ is closed the complement is an
at most countable disjoint union of intervals of the form $(s_1,S]$,
$(s_2,s_3)$, $[0,s_4)$, or $[0,S]$ which are maximal in the sense that
they cannot be extended without meeting  $\mathfrak{S}$. 
Thus, for the ``open'' sides $s_j$ this means $s_j \in \mathfrak S$.
In the first two cases this means $\sfu(s_j)=M(\sft(s_j),\sfz(s_j))$,
i.e.\ we start a gradient flow with initial condition in the global
minimizer. Hence, the solution stays constant for all future times,
i.e.\ $\sfu(s)=\sfu(s_{1,2})$ for  $s\in (s_1,S]$ or 
$(s_2,s_3)$, respectively. But this contradicts the fact that
$s\mapsto \calE(\sft(s_j),\sfu(s),\sfz(s_j))$ is strictly decreasing
(cf.\ Step 1). Hence, the first two cases cannot occur, and we
conclude $\mathfrak S=[s_*,S] $ with $s_*=s_4$ or $\mathfrak
S=\emptyset$. In particular, assertion (a) is established.

\underline{Step 3:}  To show (b) assume $s\in \mathfrak S =[s_*,S]$,
then $\sfu(s)= M(\sft(s),\sfz(s))$  by the definition of $\mathfrak S$.
Observe that $\rmD_z \calI(t,z)= \rmD_z\calE(t,M(t,z),z)+ \rmD_z
M(t,z)^T\rmD_u\calE(t,M(t,z), z)= \rmD_z\calE(t,M(t,z),z)+0$. Thus,
$(\sft,\sfz)$ solves
\[
\text{(ii)'}\quad
0\in (1 {-}\thn z)\partial\calR_0(\sfz,\sfz')+ \thn z
\bbV_\sfz \sfz'  + (1{-}\thn z) \rmD_z \calI(\sft,\sfz),\qquad
 \text{(iv)'} \quad \sft'\thn z
=0, \qquad \text{(vi)'} \quad \sft'+|\sfz'|>0,
\]
which proves that $(\sft,\sfz)$ is a BV solution of the reduced
system. For the latter relation note that $\sft'(s)+|\sfz'(s)|=0$
implies $\sfu'(s)=\frac\rmd{\rmd s} M(\sft(s),\sfz(s))=0$  so   that
(vi)' follows from (vi).
\end{proof}

Our approach in Step 1 of the above proof uses the qualitative ideas from
\cite{Zani07SPFD,AgRoSa14?TCG}, but our reduction to the simpler
convex case makes the analysis much easier. 

\paragraph{\bf \underline{Case $\alpha=1$:} comparable relaxation
  times,} Here $u$ and $z$ relax at the same rate.
%We only have the switching condition \eqref{switching}.
At a jump, the system may switch to the viscous regime $\rgm VuVz$,
where \emph{both} in the evolution of $u$, and in the evolution for
$z$, viscous dissipation intervenes, modulated by the same coefficient
$\theta = \thn u= \thn z$.
  %(i.e.\ $\frac{\thn u}{1-\thn u} = \lmname z>0$  in \eqref{diff-syst}.
\medskip

\paragraph{\bf \underline{Case $\alpha\in (0,1)$:} fast relaxation of
  $z$.} Here $z$ relaxes faster than $u$, and jumps in the
$z$-component are faster than jumps in the $u$-component.  If $z$
jumps (possibly governed by viscous dissipation), than $u$ stays
fixed,  i.e.\ $u$ is blocked while $z$ moves viscously (regime
$\rgm BuVz$). But then $u$ has still to relax to equilibrium, and
it will do it on a faster scale than  the rate-independent motion
of $z$, if $z$ stays in locally stable states (regime  $\rgm
VuRz$).  Finally, full rate-independent behavior in the regime
$\rgm EuRz$ will occur, where $\sft'(s)>0$.  Unlike in the case
$\alpha>1$,  all three  regimes may occur more than once
in the evolution of the system,  see Section \ref{ss:6.2} for an
example.

\section{Examples}
\label{s:5}

To illustrate the difference between the three  limit models (namely for $\alpha>1$, $\alpha=1$, and $\alpha \in (0,1)$),
we discuss two  examples. The first one treats a quadratic
energy and emphasizes the different initial behavior before the
solution converges to a truly rate-independent regime. In the second
example we show that solutions that start in a rate-independent
regime and coincide for the three different limit models may separate
if viscous jumps start, leading to different rate-independent behavior
afterwards.

\subsection{Initial relaxation for a system with quadratic energy}
\label{ss:6.1}

We consider the energy functional $\ene t{u,z}=\frac12(u-z)^2+ \frac12 z^2 -t u$
and trivial viscous energies leading to the ODE system
\begin{align}
  \label{eq:Exa1}
  \left\{\begin{array}{ll}
    0 = \eps^\alpha \dot u + u - z - t, \\
   0\in \Sign(\dot z) + \eps \dot z + 2 z - u
  \end{array} \right. \qquad \text{with } (u(0),z(0))=(2,-3/2).
\end{align}
We show simulations for the three cases \textcolor{blue}{$\alpha=2$
  (blue)}, \textcolor{green}{$\alpha=1$ (green)}, and
\textcolor{red}{$\alpha=1/2$ (red)} with sufficiently small $\eps$
(typically $0.001 \ldots 0.03$). The components $u$ and $z$ as
functions of time are depicted in Figure \ref{F1:uz}.

 However, to  detect different jump behavior at $t\approx 0$ it
is advantageous to look at the parameterized solutions, which are
depicted in Figure \ref{F3:Para}, showing 
$(\sft,\sfq)$  for the three different cases. The parameterization was
calculated using $\dot s(t)=\max\{0.5,|\dot u(t)|,\dot z(t)| \}$.
\begin{figure}
\centering
\includegraphics[width=16em]{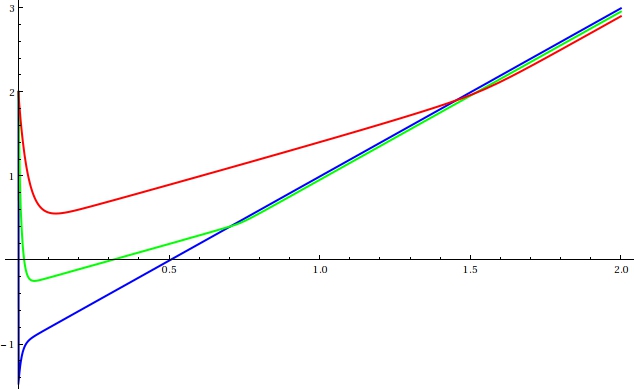}
{\unitlength1em
\begin{picture}(0,0)
\put(-15,9){$u(t)$}
\put(-1,3.8){$t$}
\end{picture}}%
\qquad
\includegraphics[width=16em]{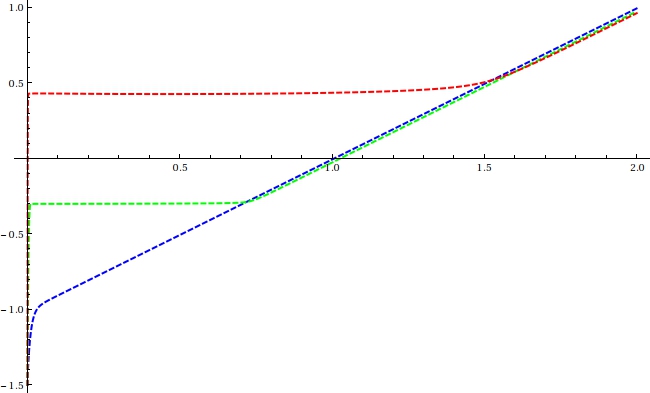}
{\unitlength1em
\begin{picture}(0,0)
\put(-15,9){$z(t)$}
\put(-1,4.5){$t$}
\end{picture}}
\caption{Solutions for \eqref{eq:Exa1}
  for the three cases \textcolor{blue}{$\alpha=2$
  (blue)}, \textcolor{green}{$\alpha=1$ (green)}, and
\textcolor{red}{$\alpha=1/2$ (red)}.}
\label{F1:uz}
\end{figure}
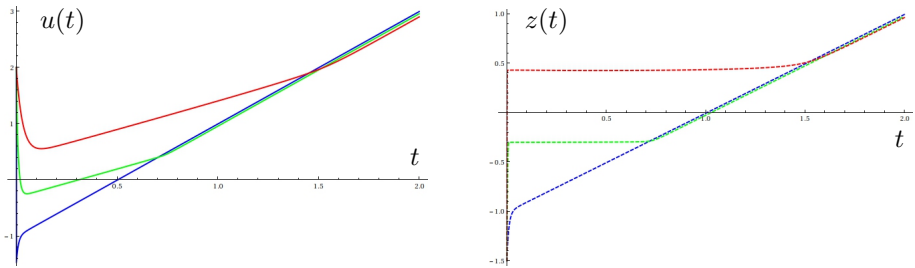
\begin{figure}
\centering {\unitlength1cm
\includegraphics[width=5\unitlength]{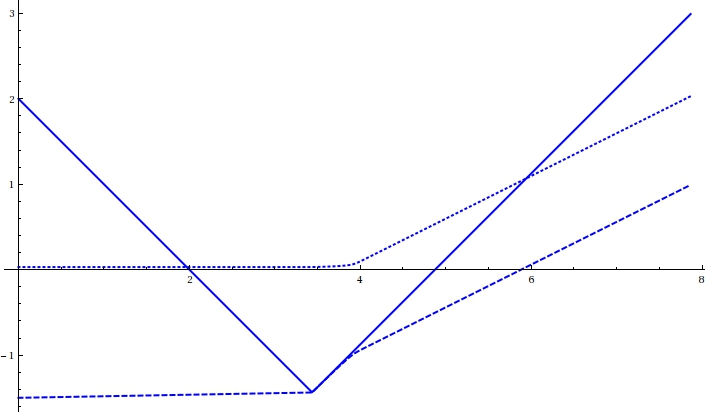}%
\begin{picture}(0,0)(5,0)
\put(-0.3,2.2){$\sfu$} \put(-0.3,1){$\sft$} \put(-0.3,0){$\sfz$}
\put(2.07,0){\line(0,1){2.7}} \put(2.45,0){\line(0,1){2.7}}
\color{blue}\put(0.7,2.5){$\rgm VuBz$}
\put(1.85,2.8){$\rgm EuVz$} \put(3,2.5){$\rgm EuRz$}
\end{picture}
\quad
\includegraphics[width=5\unitlength]{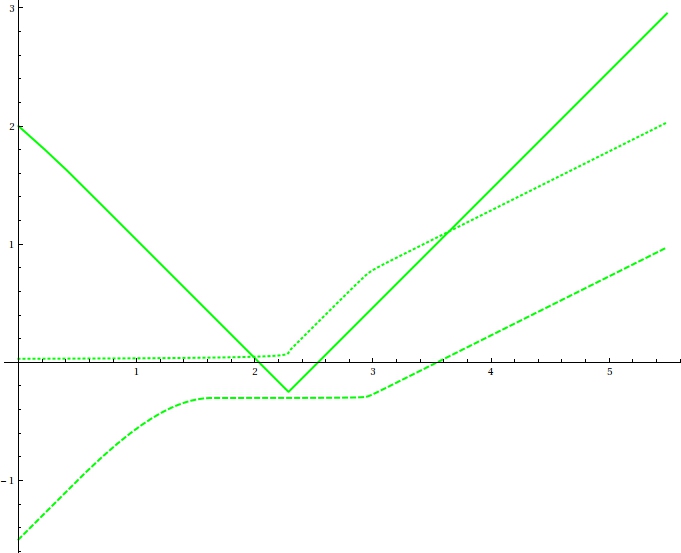}%
\begin{picture}(0,0)(5,0)
\put(-0.2,3.1){$\sfu$} \put(-0.2,1.3){$\sft$} \put(-0.2,0){$\sfz$}
\put(2.1,0){\line(0,1){3.7}}
\color{green}\put(0.7,3.3){$\rgm VuVz$}
   \put(2.6,3.3){$\rgm EuRz$}
\end{picture}
\quad
\includegraphics[width=5\unitlength]{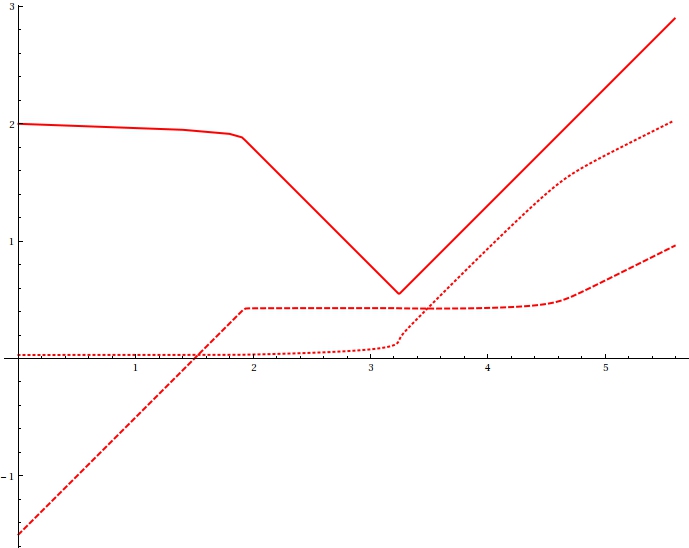}%
\begin{picture}(0,0)(5,0)
\put(-0.2,3){$\sfu$} \put(-0.2,1.3){$\sft$} \put(-0.15,0){$\sfz$}
\put(1.8,0.5){\line(0,1){3.2}} \put(2.85,0.5){\line(0,1){3.2}}
\color{red}\put(0.7,3.5){$\rgm BuVz$}
\put(1.9,3.5){$\rgm VuRz$} \put(3.3,3.5){$\rgm EuRz$}
\end{picture}
}
\caption{Solutions $(\sft,\sfu,\sfz)$ for \eqref{eq:Exa1} with
 dotted $\sft$, full $\sfu$, and dashed $\sfz$. \color{blue} Left
  $\alpha=2$, \color{green} middle $\color{green}\alpha=1$,
  \color{red} right $\alpha=1/2$.} 
\label{F3:Para}
\end{figure}
 In the parameterized form we fully see the structure of the jump
for $t\approx 0$. For \underline{$\alpha=2$} we
obtain first a jump from the initial datum $(u,z)=(2,-1.5)$ to
$(u,z)=(-1.5,-1.5)$ on the timescale $\eps^2$, which is the regime
$\rgm VuBz$. Then, $u$ is equilibrated, and a jump to $(-1,-1)$ along
the diagonal $u=z$ occurs on the timescale $\eps$, which is the regime
$\rgm EuVz$. Finally,  the
solution finds the rate-independent regime $\rgm EuRz$ with
$(u(t),z(t))=q_\mathrm{ri}(t):=(2t{-}1,t{-}1)$.

For \underline{$\alpha=1/2$} the
solution first jumps to $(2,0.5)$ on the time scale $\eps$, which is
the regime $\rgm BuVz$. Next,  and then there is a jump to
$(0.5,0.5)$ in the time scale $\eps^{1/2}$, which is regime $\rgm
VuRz$. Then, the rate-independent regime $\rgm EuRz$ starts, namely
via $(u(t),z(t))=(t{-}0.5, 0.5)$ for $t\in {]0,1.5]}$  and
$q_\mathrm{ri}$ for $t>1.5$.

The behavior for \underline{$\alpha=1$} is
intermediate: the jump occurs along a nonlinear curve  in regime
$\rgm VuVz$,  and
$q_\mathrm{ri}$ is joined for $t\geq t_*\approx 0.7$,  which is regime
$\rgm EuRz$. 

The different behavior and the different regimes are
also nicely seen by plotting the trajectories in the $(u,z)$-plane,
see Figure \ref{F2:uzPlane},  where the three different cases for
$\alpha$ are depicted again. 
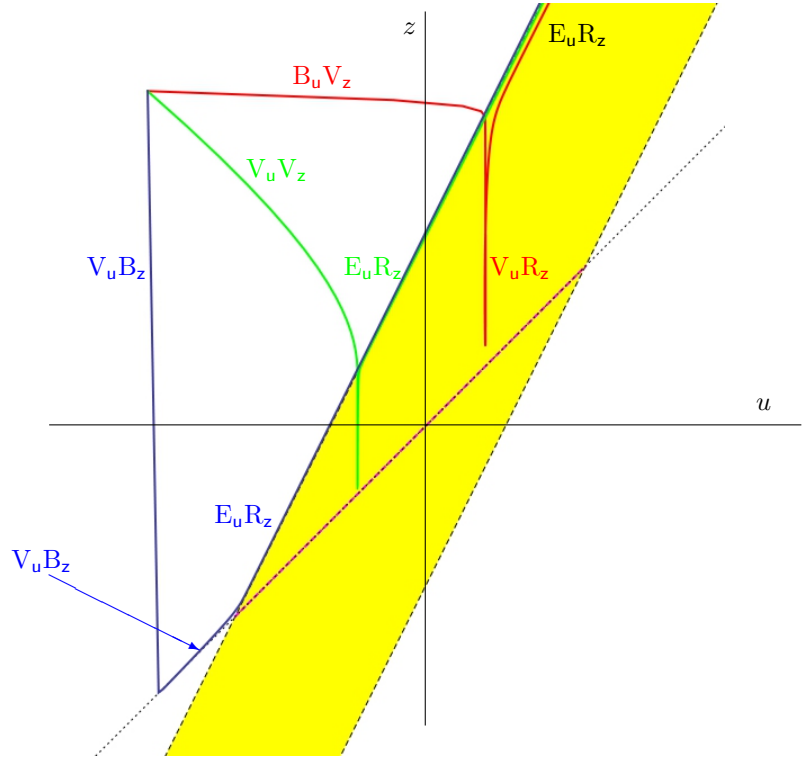
\begin{figure}
{\unitlength1cm
\centering \includegraphics[width=10cm]{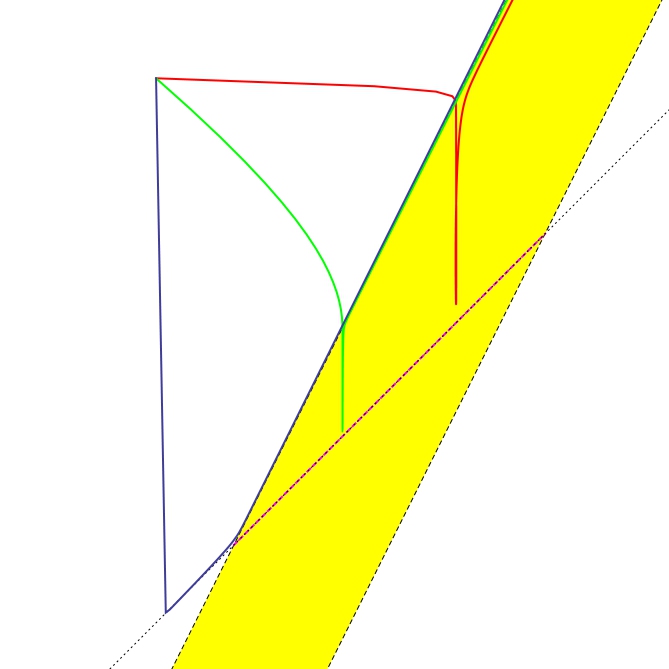}
\begin{picture}(0,0)(4.1,-4.4)
\put(4.4,0.2){$u$} \put(-0.3,5.2){$z$}
\put(-5,0){\line(1,0){10}}
\put(0,-4){\line(0,1){9.5}}
{\color{blue}\put(-4.5,2){$\rgm VuBz$}
   \put(-5,-2){\vector(2,-1){2}}
   \put(-5.5,-1.9){$\rgm VuBz$}
   \put(-2.8,-1.3){$\rgm EuRz$}}
{\color{green}\put(-2.5,3.3){$\rgm VuVz$}
   \put(-1.2,2){$\rgm EuRz$}}
{\color{red}\put(-2,4.53){$\rgm BuVz$}
   \put(0.62,2){$\rgm VuRz$}}
\put(1.4,5.1){$\rgm EuRz$}
\end{picture}}
\caption{{Solutions $(z(t),u(t))$ for \eqref{eq:Exa1}. The dotted line
  is the diagonal $u=z$, while the yellow area is the  locally
  stable  region $|2z{-}u|\leq 1$.} }
\label{F2:uzPlane}
\end{figure}

\subsection{Different jumps starting from the rate-independent regime}
\label{ss:6.2}

 Finally we provide  an example where the jumps start out of a
rate-independent motion,  i.e.\ we first have the regime $\rgm
EuRz$, and then the system becomes unstable and develops a jump. For
this purpose we use the nonconvex  energy
\begin{align*}
&\ene t{u,z}=\frac12 (u{-}g(z))^2 + F(z) - tu \quad \text{with }
g(z)=4z^3-4z\\
&\text{ and } F'(z)=-1 + (z{+}1)^2\big({-}40+10(z{+}1)^2+
38 \mathrm{e}^{-10(z+0.5)^2} \big).
\end{align*}
Using the standard viscous potentials as above,  the ODE system
reads
\begin{align}
  \label{eq:Exa2}
  \left\{
  \begin{array}{ll}
0 = \eps^\alpha \dot u + u - g(z) - t,\\
0\in \Sign(\dot z) + \eps \dot z + F'(z) +g'(z)(g(z){-}u)
\end{array}
\right. \quad \text{with }
(u(-0.2),z(-0.2))= (-2.4,-1.2).
\end{align}

\begin{figure}
\centering \includegraphics[width=16em]{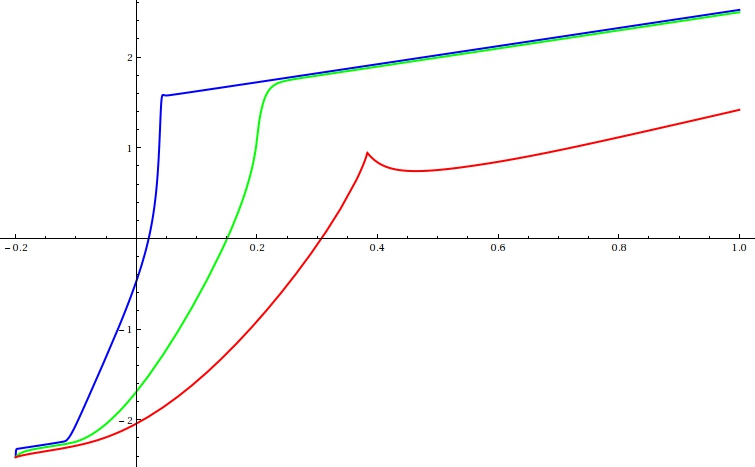} \qquad
\includegraphics[width=16em]{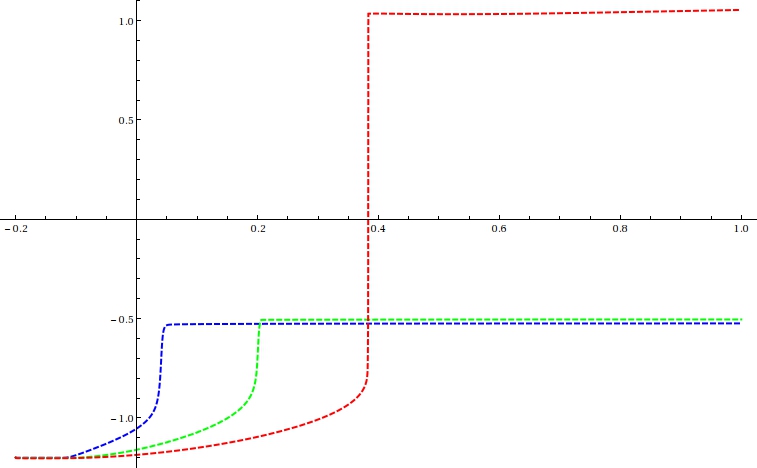}
\caption{Solutions for \eqref{eq:Exa2}: left $u(t)$  and right $z(t)$}
\label{F4:uz}
\end{figure}
Figure \ref{F4:uz} shows simulation results of $u(t)$ and $z(t)$ for
the three cases \textcolor{blue}{$\alpha=2$
  (blue)}, \textcolor{green}{$\alpha=1$ (green)}, and
\textcolor{red}{$\alpha=1/2$ (red)} with sufficiently small $\eps$. We see
that the solutions stay together for $t\in [-0.2,-0.1]$,  which is
exactly the time they stay in regime $\rgm EuRz$. Then, in all three
cases a jump develops, but this is quite different for different
$\alpha$. In Figure \ref{F6:Para} we provide graphics of the same
solutions, but now in the reparameterized form $(\sft,\sfu,\sfz)$ for
the three $\alpha$-values $2,\ 1$, and $1/2$, where again the
parameterization $s$ is chosen such that  $\dot
s(t)=\max\{0.5,|\dot u(t)|,\dot z(t)| \}$.  However, for this
example numerical instabilities prevented us from taking $\eps$ small
enough to have a better separation of time scale. Even in the
viscous regimes we still see $\sft'>0$ but small. Nevertheless,
Figure \ref{F6:Para} clearly shows the different regimes.
\begin{figure}
\centering
{\unitlength1cm
 \includegraphics[width=5\unitlength]{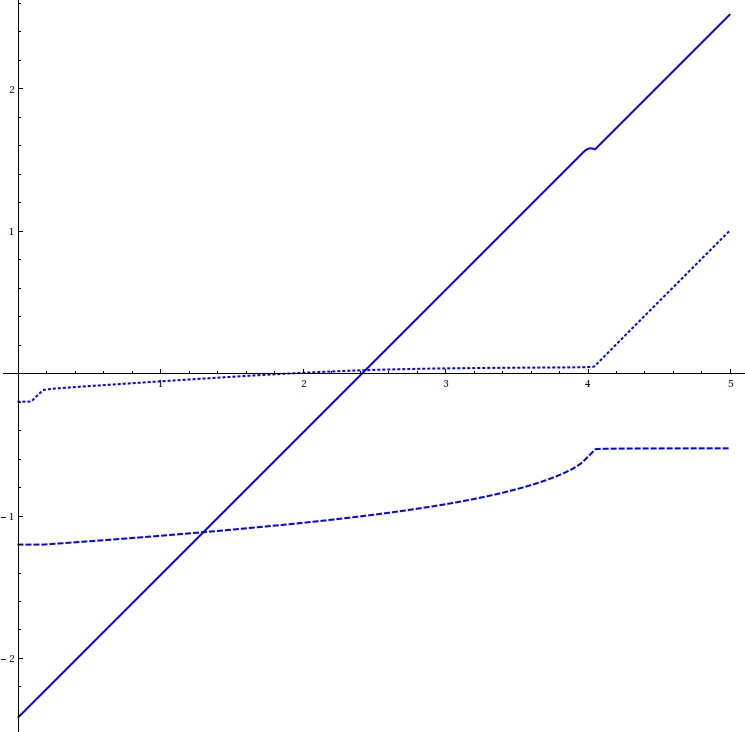}%
 \begin{picture}(0,0)(5,0)
  \put(-0.2,0){$\sfu$} \put(-0.2,1.2){$\sfz$} \put(-0.2,2.1){$\sft$}
  \put(0.5,0){\line(0,1){4}}  \put(3.99,1){\line(0,1){4}}
  \color{blue} \put(-0.1,4){$\rgm EuRz$}  \put(1.7,4){$\rgm EuVz$}
  \put(4.3,3.64){$\rgm EuRz$}
 \end{picture}
 \quad
 \includegraphics[width=5\unitlength]{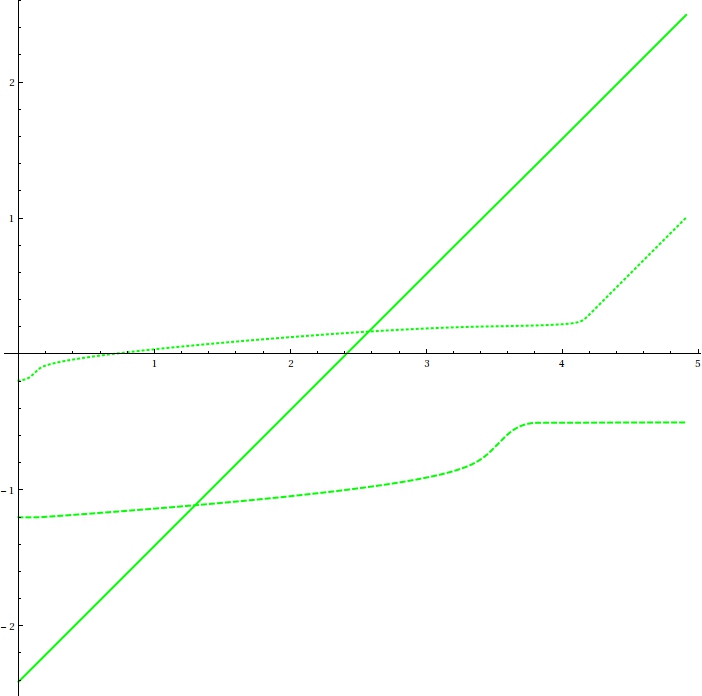}%
 \begin{picture}(0,0)(5,0)
  \put(-0.2,0){$\sfu$} \put(-0.2,1.2){$\sfz$} \put(-0.2,2.1){$\sft$}
  \put(0.5,0){\line(0,1){4}}  \put(3.9,1){\line(0,1){4}}
  \color{green} \put(-0.1,4){$\rgm EuRz$}  \put(1.7,4){$\rgm VuVz$}
  \put(4.2,3.64){$\rgm EuRz$}
 \end{picture}
 \quad
 \includegraphics[width=5\unitlength]{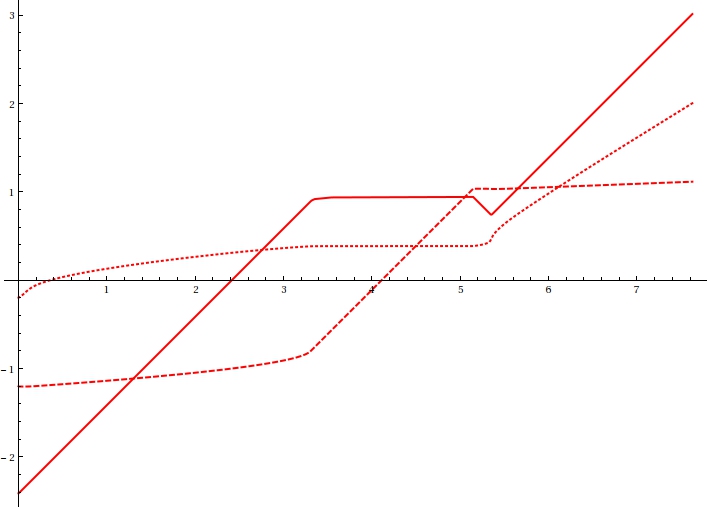}%
 \begin{picture}(0,0)(5,0)
  \put(-0.2,0){$\sfu$} \put(-0.2,0.8){$\sfz$} \put(-0.2,1.5){$\sft$}
  \put(0.5,0){\line(0,1){3}} \put(2.2,0){\line(0,1){3}}
  \put(3.3,0){\line(0,1){2.9}} \put(3.5,0){\line(0,1){2.9}}
  \color{red} \put(-0.1,3.5){$\rgm EuRz$}  \put(1,3){$\rgm VuRz$}
  \put(2.35,3.5){$\rgm BuVz$} \put(3,3){$\rgm VuRz$} \put(3.8,3.5){$\rgm EuRz$}
 \end{picture}
}
\caption{Solutions $(\sft,\sfu,\sfz)$ for \eqref{eq:Exa2} with
 dotted $\sft$, full $\sfu$, and dashed $\sfz$. \color{blue} Left
  $\alpha=2$, \color{green} middle $\color{green} \alpha=1$,
  \color{red} right $\alpha=1/2$.} 
\label{F6:Para}
\end{figure}

Figure \ref{F5:uzPlane} shows the trajectories in
the $(z,u)$-plane.
\begin{figure}
\centering \includegraphics[width=26em]{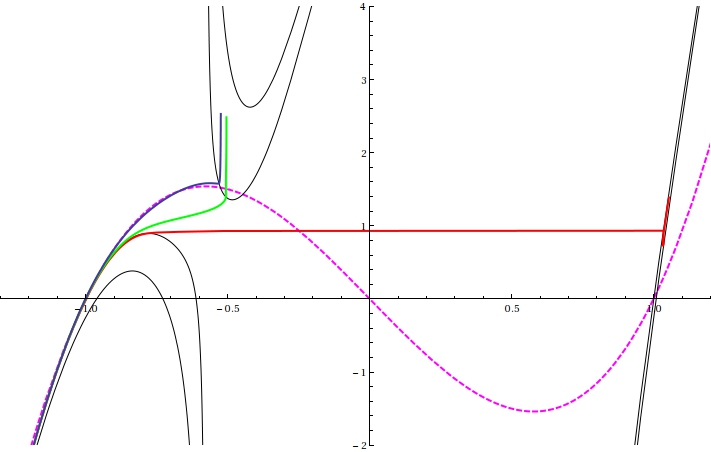}
\caption{Solutions $(z(t),u(t))$ for \eqref{eq:Exa2}. The dashed magenta line
  is $u=g(z)$, while the black curves display the boundaries of the
  locally stable domain   $|F'(z) +g'(z)(g(z){-}u)| \leq 1$. }
\label{F5:uzPlane}
\end{figure}

\bibliographystyle{my_alpha}
\bibliography{refs_2visc}

\def\cprime{$'$} \def\cprime{$'$}
\begin{thebibliography}{AA00}\itemsep0.1em

\bibitem[AGS08]{AGS08}
{\scshape L.~Ambrosio, N.~Gigli, {\upshape and} G.~Savar{\'e}}.
\newblock {\em Gradient flows in metric spaces and in the space of probability
  measures}.
\newblock Lectures in Mathematics ETH Z\"urich. Birkh\"auser Verlag, Basel,
  second edition, 2008.

\bibitem[Amb95]{Ambrosio95}
{\scshape L.~Ambrosio}.
\newblock Minimizing movements.
\newblock {\em Rend. Accad. Naz. Sci. XL Mem. Mat. Appl. (5)}, 19, 191--246,
  1995.

\bibitem[ARS14]{AgRoSa14?TCG}
{\scshape V.~Agostiniani, R.~Rossi, {\upshape and} G.~Savar\'e}.
\newblock Balanced viscosity solutions of singularly perturbed gradient flows
  in infinite dimension.
\newblock {\em In preparation}, 2014.

\bibitem[BFM12]{BabFraMor12}
{\scshape J.-F.~Babadjian, G.~Francfort, {\upshape and} M.~Mora}.
\newblock Quasistatic evolution in non-associative plasticity - the cap model.
\newblock {\em SIAM J. Math. Anal.}, 44, 245--292, 2012.

\bibitem[Col92]{Colli92}
{\scshape P.~Colli}.
\newblock On some doubly nonlinear evolution equations in {B}anach spaces.
\newblock {\em Japan J. Indust. Appl. Math.}, 9(2), 181--203, 1992.

\bibitem[CoV90]{ColliVisintin90}
{\scshape P.~Colli {\upshape and} A.~Visintin}.
\newblock On a class of doubly nonlinear evolution equations.
\newblock {\em Comm. Partial Differential Equations}, 15(5), 737--756, 1990.

\bibitem[{Da}S13]{DM-Scala}
{\scshape G.~{Dal Maso} {\upshape and} R.~Scala}.
\newblock Quasistatic evolution in perfect plasticity as limit of dynamic
  processes.
\newblock {\em Preprint}, 2013.
\newblock Available at \texttt{http://cvgmt.sns.it/paper/2206/}.

\bibitem[DDS11]{DalMaso-DeSimone-Solombrino2011}
{\scshape G.~{Dal Maso}, A.~{DeSimone}, {\upshape and} F.~{Solombrino}}.
\newblock Quasistatic evolution for {C}am-{C}lay plasticity: a weak formulation
  via viscoplastic regularization and time rescaling.
\newblock {\em Calc. Var. Partial Differential Equations}, 40(1-2), 125--181,
  2011.

\bibitem[DFT05]{DaFrTo05QCGN}
{\scshape G.~{Dal Maso}, G.~Francfort, {\upshape and} R.~Toader}.
\newblock Quasistatic crack growth in nonlinear elasticity.
\newblock {\em Arch. Rational Mech. Anal.}, 176, 165--225, 2005.

\bibitem[DMDS12]{DalMaso-DeSimone-Solombrino2012}
{\scshape G.~Dal~Maso, A.~DeSimone, {\upshape and} F.~Solombrino}.
\newblock Quasistatic evolution for {C}am-{C}lay plasticity: properties of the
  viscosity solution.
\newblock {\em Calc. Var. Partial Differential Equations}, 44(3-4), 495--541,
  2012.

\bibitem[EfM06]{ef-mie06}
{\scshape M.~Efendiev {\upshape and} A.~Mielke}.
\newblock On the rate--independent limit of systems with dry friction and small
  viscosity.
\newblock {\em J. Convex Analysis}, 13(1), 151--167, 2006.

\bibitem[FrS13]{FraSte}
{\scshape G.~A.~Francfort {\upshape and} U.~Stefanelli}.
\newblock Quasi-static evolution for the {A}rmstrong-{F}rederick
  hardening-plasticity model.
\newblock {\em Appl. Math. Res. Express. AMRX}, 2013(2), 297--344, 2013.

\bibitem[Iof77]{Ioff77LSIF}
{\scshape A.~D.~Ioffe}.
\newblock On lower semicontinuity of integral functionals. {I}.
\newblock {\em SIAM J. Control Optimization}, 15(4), 521--538, 1977.

\bibitem[IoT79]{IofTih}
{\scshape A.~D.~Ioffe {\upshape and} V.~M.~Tihomirov}.
\newblock {\em Theory of extremal problems}, volume~6 of {\em Studies in
  Mathematics and its Applications}.
\newblock North-Holland Publishing Co., Amsterdam-New York, 1979.
\newblock Translated from the Russian by Karol Makowski.

\bibitem[KMZ08]{KnMiZa07?ILMC}
{\scshape D.~Knees, A.~Mielke, {\upshape and} C.~Zanini}.
\newblock On the inviscid limit of a model for crack propagation.
\newblock {\em Math. Models Methods Appl. Sci.}, 18(9), 1529--1569, 2008.

\bibitem[KRZ13]{KnRoZa2011}
{\scshape D.~{Knees}, R.~{Rossi}, {\upshape and} C.~{Zanini}}.
\newblock A vanishing viscosity approach to a rate-independent damage model.
\newblock {\em Math. Models Methods Appl. Sci.}, 23(4), 565--616, 2013.

\bibitem[Mie03]{Miel03EFME}
{\scshape A.~Mielke}.
\newblock Energetic formulation of multiplicative elasto--plasticity using
  dissipation distances.
\newblock {\em Cont. Mech. Thermodynamics}, 15, 351--382, 2003.

\bibitem[Mie05]{Miesurvey}
{\scshape A.~Mielke}.
\newblock Evolution in rate-independent systems ({C}h.~6).
\newblock In C.~Dafermos {\upshape and} E.~Feireisl, editors, {\em Handbook of
  Differential Equations, Evolutionary Equations, vol.~2}, pages 461--559.
  Elsevier B.V., Amsterdam, 2005.

\bibitem[Mie11]{Miel08?DEMF}
{\scshape A.~Mielke}.
\newblock Differential, energetic and metric formulations for rate-independent
  processes ({C}h.~3).
\newblock In L.~Ambrosio {\upshape and} G.~Savar\'e, editors, {\em Nonlinear
  PDEs and Applications.C.I.M.E. Summer School, Cetraro, Italy 2008}, pages
  87--170. Springer, Heidelberg, 2011.

\bibitem[MiT99]{Mie-Theil99}
{\scshape A.~Mielke {\upshape and} F.~Theil}.
\newblock A mathematical model for rate-independent phase transformations with
  hysteresis.
\newblock In H.-D.~Alber, R.~Balean, {\upshape and} R.~Farwig, editors, {\em
  Proceedings of the Workshop on ``Models of Continuum Mechanics in Analysis
  and Engineering''}, pages 117--129, Aachen, 1999. Shaker-Verlag.

\bibitem[MiT04]{Mielke-Theil04}
{\scshape A.~Mielke {\upshape and} F.~Theil}.
\newblock On rate--independent hysteresis models.
\newblock {\em Nonl. Diff. Eqns. Appl. (NoDEA)}, 11, 151--189, 2004.
\newblock (Accepted July 2001).

\bibitem[MiZ14]{Mielke-Zelik}
{\scshape A.~Mielke {\upshape and} S.~Zelik}.
\newblock On the vanishing viscosity limit in parabolic systems with
  rate-independent dissipation terms.
\newblock {\em Ann. Sc. Norm. Sup. Pisa Cl. Sci. (5)}, XIII, 67--135, 2014.

\bibitem[MRS09]{MRS09}
{\scshape A.~Mielke, R.~Rossi, {\upshape and} G.~Savar\'e}.
\newblock Modeling solutions with jumps for rate-independent systems on metric
  spaces.
\newblock {\em Discrete Contin. Dyn. Syst.}, 25, 585--615, 2009.

\bibitem[MRS12]{MRS10}
{\scshape A.~Mielke, R.~Rossi, {\upshape and} G.~Savar\'e}.
\newblock Bv solutions and viscosity approximations of rate-independent
  systems.
\newblock {\em ESAIM Control Optim. Calc. Var.}, 18, 36--80, 2012.

\bibitem[MRS13a]{mielke-rossi-savare2013}
{\scshape A.~Mielke, R.~Rossi, {\upshape and} G.~Savar\'e}.
\newblock Balanced viscosity (bv) solutions to infinite-dimensional
  rate-independent systems.
\newblock {\em J. Eur. Math. Soc.}, 2013.
\newblock Submitted. WIAS preprint 1845.

\bibitem[MRS13b]{MRS-dne}
{\scshape A.~Mielke, R.~Rossi, {\upshape and} G.~Savar\'e}.
\newblock Nonsmooth analysis of doubly nonlinear equations.
\newblock {\em Calc. Var. Partial Differential Equations}, 46(1-2), 253--310,
  2013.

\bibitem[Rac12]{Racca}
{\scshape S.~Racca}.
\newblock A viscosity-driven crack evolution.
\newblock {\em Adv. Calc. Var.}, 5(4), 433--483, 2012.

\bibitem[Rou09]{Roub09}
{\scshape T.~Roub{\'{\i}}{\v{c}}ek}.
\newblock Rate-independent processes in viscous solids at small strains.
\newblock {\em Math. Meth. Appl. Sci.}, 32, 825--862, 2009.

\bibitem[Rou10]{Roub10}
{\scshape T.~Roub{\'{\i}}{\v{c}}ek}.
\newblock Thermodynamics of rate-independent processes in viscous solids at
  small strains.
\newblock {\em SIAM J. Math. Anal.}, 40, 256--297, 2010.

\bibitem[Rou13]{Roub13}
{\scshape T.~Roub{\'{\i}}{\v{c}}ek}.
\newblock Adhesive contact of visco-elastic bodes and defect measures arising
  by vanishing viscosity.
\newblock {\em SIAM J. Math. Anal.}, 45, 101--126, 2013.

\bibitem[Sca14]{Scala}
{\scshape R.~Scala}.
\newblock Limit of viscous dynamic processes in delamination as the viscosity
  and inertia vanish.
\newblock {\em Preprint}, 2014.
\newblock Available at \texttt{http://cvgmt.sns.it/paper/2434/}.

\bibitem[ToZ09]{ToaZan06}
{\scshape R.~Toader {\upshape and} C.~Zanini}.
\newblock An artificial viscosity approach to quasistatic crack growth.
\newblock {\em Boll. Unione Mat. Ital. (9)}, 2(1), 1--35, 2009.

\bibitem[Zan07]{Zani07SPFD}
{\scshape C.~Zanini}.
\newblock Singular perturbations of finite dimensional gradient flows.
\newblock {\em Discr. Cont.\ Dyanm. Syst.}, 18(4), 657--675, 2007.

\end{thebibliography}

\end{document}